\PassOptionsToPackage{dvipsnames}{xcolor}

\documentclass[10pt]{article}

\usepackage{amsfonts,amssymb,amsmath,amsopn,float}
\usepackage{amsthm}
\usepackage{comment}
\usepackage{xifthen}
\usepackage{graphicx,svg}
\usepackage{tikz}
\usepackage{pgfplots}
\usepgfplotslibrary{units}
\usepackage{epstopdf}
\usepackage[utf8]{inputenc}
\usepackage{mathtools}
\usepackage{url}
\usepackage{subcaption}
\usepackage{MnSymbol}
\usepackage{moresize}
\usetikzlibrary{patterns}

\usepackage{cleveref}

     \makeatletter
\def\@fnsymbol#1{\ensuremath{\ifcase#1\or \dagger\or *\or \ddagger\or
     \mathsection\or \mathparagraph\or \|\or **\or \dagger\dagger
     \or \ddager\ddager \else\@cterr\fi}}
     \makeatother
\newcommand{\bzz}{{\boldsymbol{\zeta}}}

\newcommand{\Ltwo}[1]{%
	\ifthenelse{\equal{#1}{}}{L^2}{L^2(#1)}%
}

\newcommand{\Ltwoz}[1]{%
	\ifthenelse{\equal{#1}{}}{L^2_0}{L^2_0(#1)}%
}

\newcommand{\Cone}[1]{%
	\ifthenelse{\equal{#1}{}}{C^{1}}{C^{1}(#1)}%
}

\newcommand{\Conez}[1]{%
	\ifthenelse{\equal{#1}{}}{C^{1}_{0}}{C^{1}_{0}(#1)}%
}

\newcommand{\Ctwo}[1]{%
	\ifthenelse{\equal{#1}{}}{C^{2}}{C^2(#1)}%
}

\newcommand{\Ctwoz}[1]{%
	\ifthenelse{\equal{#1}{}}{C^{2}_{0}}{C^{2}_{0}(#1)}%
}

\newcommand{\Cholder}[1]{%
	\ifthenelse{\equal{#1}{}}{C^{0,\gamma}}{C^{0,\gamma}(#1)}%
}

\newcommand{\Cholderz}[1]{%
	\ifthenelse{\equal{#1}{}}{C^{0,\gamma}_{0}}{C^{0,\gamma}_{0}(#1)}%
}

\newcommand{\bolds}[1]{\boldsymbol{#1}}

\newcommand{\bb}{\bolds{b}}

\newcommand{\be}{\bolds{e}}

\newcommand{\bg}{\bolds{g}}

\newcommand{\bu}{\bolds{u}}
\newcommand{\bv}{\bolds{v}}

\newcommand{\bx}{\bolds{x}}
\newcommand{\by}{\bolds{y}}

\newcommand{\brx}{{\bar{\bx}}}
\newcommand{\bry}{{\bar{\by}}}
\newcommand{\breps}{{\bar{\epsilon}}}

\newcommand{\brmcL}{{\bar{\mathcal{L}}}}
\newcommand{\bru}{{\bar{\bu}}}

\newcommand{\brt}{{\bar{t}}}
\newcommand{\brr}{{\bar{r}}}

\newcommand{\brS}{{\bar{S}}}

\newcommand{\brC}{{\bar{C}}}
\newcommand{\brJ}{{\bar{J}}}
\newcommand{\brT}{{\bar{T}}}
\newcommand{\brbeta}{{\bar{\beta}}}
\newcommand{\brrho}{{\bar{\rho}}}

\newcommand{\sautoref}[2]{\hyperref[#2]{#1 \ref*{#2}}}

\newtheorem{theorem}{Theorem}[section]
\newtheorem{lemma}{Lemma}[section]

\newtheorem{definition}{Definition}[section]
\numberwithin{equation}{section}

\newcommand{\ee}{\mathbf{e}}
\newcommand{\xx}{\mathbf{x}}
\newcommand{\yy}{\mathbf{y}}

\newcommand{\nn}{\mathbf{n}}
\newcommand{\uu}{\mathbf{u}}

\newcommand{\vv}{\mathbf{v}}
\newcommand{\ww}{\mathbf{w}}
\newcommand{\zz}{\mathbf{z}}

\newcommand{\LL}{\mathcal{L}}

\title{Energy balance and damage for dynamic brittle fracture from a nonlocal formulation}

\author{Robert  P. Lipton  \thanks{Department of Mathematics, Louisiana State University,
              Baton Rouge, LA 70803,
              Orcid: https://orcid.org/0000-0002-1382-3204,
              \tt{lipton@lsu.edu}}
\and Debdeep Bhattacharya \thanks{Department of Mathematics, University of Utah,
	      Salt Lake City, UT 84112,
	      Orcid:  https://orcid.org/0000-0001-5171-5506, 
	      \tt{ d.Bhattacharya@utah.edu}}
 }

\date{}

\begin{document}

\maketitle

\begin{abstract}
A  nonlocal model of peridynamic type for dynamic brittle damage  is introduced consisting of two phases, one elastic and the other inelastic.  Evolution from the elastic to the inelastic phase depends on material strength.  Existence and uniqueness of the displacement-failure set  pair  follow from  the initial value problem.  The displacement-failure pair satisfies energy balance. The length of nonlocality $\epsilon$ is taken to be small relative to the domain in $\mathbb{R}^d$, $d=2,3$. The new nonlocal model delivers a two point strain evolution on a subset of $\mathbb{R}^d\times\mathbb{R}^d$. This evolution provides an energy that interpolates between volume energy corresponding to elastic behavior and surface  energy corresponding to failure. In general the deformation energy resulting in material failure over a region $R$ is  given by a $d-1$ dimensional integral that is uniformly bounded as $\epsilon\rightarrow 0$.  For fixed $\epsilon$, the failure energy is nonzero for $d-1$ dimensional regions $R$ associated with flat crack surfaces.   This failure energy is the Griffith fracture energy  given by the energy release rate  multiplied by  area for $d=3$ (or length for $d=2$).  The nonlocal field theory is shown to recover a solution of Naiver's equation outside a propagating flat traction free crack  in  the limit of vanishing spatial nonlocality.  
 Simulations illustrate fracture evolution through generation of an internal traction free boundary  as a wake left behind a moving strain concentration. Crack paths are seen to follow a maximal strain energy density criterion.
\end{abstract}

\section{Introduction}

Peridynamic models
implicitly couple the evolution of damage and deformation inside a material specimen through a nonlocal formulation using force interaction between neighboring points \cite{silling2000reformulation,silling2007peridynamic}.  They provide for the spontaneous emergence and growth of fissures as part of the dynamic simulation.  The time evolution of peridynamic simulations are driven by temporally and spatially nonlocal forces. This idea  has been adapted and expanded  and the literature is now very large, see  for example, the contributions \cite{sillingaskari2005, BobaruZhang, Foster, Baselevs,Zaccariotto2015,hu2016bond2,jafarzadeh2021general,hu2018thermomechanical,BhattacharyaDayal2006,xu2018elastic, Otokirkus, DuTaoTian, Seleson},  books and reviews \cite{bobaru2016handbook,javili2019peridynamics,isiet2021review,diehl2022comparative}.

In this article we rigorously pursue the free discontinuity problem for fracture mechanics  and propose a new field theory that demonstrably preserves energy balance to discover new advantages to the nonlocal approach.
The field theory describes a material undergoing irreversible damage guaranteeing existence of solution, energy balance and an explicit formula for the energy necessary for material failure. For flat cracks the formula is the classic Griffith fracture energy given by the product of the crack area and  the critical energy release rate.  Both of these results are shown to follow directly from the evolution equation for the deformation multiplied by the velocity and integration by parts. This provides for the first time,  energy balance and a  direct relationship between Griffith fracture energy and nonlocal formulations of dynamic fracture with irrevocable damage. In the limit of vanishing non-locality, the nonlocal field theory is shown to recover  the elastic displacement field as a solution of the linear elastic wave equation outside a propagating traction free crack.
The new nonlocal formulation draws motivation from  \cite{silling2000reformulation},  the existence theory of \cite{EmmrichPhust,DuTaoTian,lipton2016cohesive} and the rate form of energy balance introduced in \cite{lipton2018free}. 

A  small deformation model for brittle failure under tensile loading is developed. Forces between pairs of points in $\Omega$  are referred to as bonds. Bond forces depend on the strain given by the deformation between nearby points. The strain is referred to as two point strain since it depends on the displacement  at two different  points. The model is nonlinear and the force between pairs of points act elastically against compressive strain and for moderate tensile strain the force is linear elastic. As one continues to increase tensile strain it becomes nonlinear elastic and at a critical strain the force becomes unstable and softens with increasing strain. The force eventually goes to zero with increasing strain and the bond between points breaks.
This process is irreversible and the bonds once broken do not heal. 
In this model the maximum length scale of nonlocal interaction  is  both finite and small relative to the size of the domain and is denoted by $\epsilon$.  The failure set $\Gamma^\epsilon(t)$ is the set of pairs of points with broken bonds in $\Omega$ at time $t$,  see Section \ref{sec:nonlocal:model}.  The bond force is coupled to a two point phase field that is a function of the two point strain. The phase field is $1$ in the undamaged material and $0$ in the fully damaged material. The coupling between bond force and phase field is given by their product. When the phase field becomes zero the bond is fully damaged and this state is irreversible.

This level of generalization together with Newton's second law given by the evolution equation  couples elastic forces and failure allowing failure sets and deformation to emerge from the dynamics. 
In addition to existence, the model provides  energy balance. The rate form of energy balance is shown to follow directly from the evolution equation for the deformation multiplied by the velocity and integration by parts. The rate form of 
energy balance shows that damage must start occurring when the energy input to the system exceeds the material's ability to generate kinetic and elastic energy through displacement and velocity, see Section \ref{energybalanceprocesszone}. 
The  energy expended up to time $t$ resulting in material failure over a region $\Gamma^\epsilon(t)$ is given by a bounded $d-1$ dimensional geometric integral of the failure set.  Application of Gronwall's inequality shows that the geometric integral is bounded uniformly in $\epsilon$ for initial and  boundary conditions that are independent of $\epsilon$. 

As an example fix $\epsilon>0$ and consider the failure set $\Gamma^\epsilon(t)$ defined by a flat two dimensional piece of surface $R_t$  where points  above the surface are no longer influenced by forces due to points below the surface and vice versa. This  is the case of alignment, i.e., all bonds connecting points $\yy$  above  $R_t$ to points $\xx$ below are broken. Calculation of the  failure energy of $\Gamma^\epsilon(t)$  shows that   it is  the product of the critical energy release rate of fracture mechanics multiplied by the two dimensional surface measure of $R_t$, see Section \ref{Flatcrack}.  The surface $R_t$ defines an internal boundary to domain $\Omega$ and  the crack is unambiguously described as the internal boundary.  Displacement jumps can only occur across $R_t$ and traction forces are zero on either side of $R_t$. Additionally, the analysis of Section \ref{Flatcrack} shows that material failure is associated with a maximum energy dissipation condition on each bond. 

The example given above shows that the failure energy corresponds to Griffith fracture energy for flat cracks. 
This demonstrates that the failure energy is bounded and can be nonzero on $d-1$ dimensional sets corresponding to cracks.  The failure energy for rectifiable curved smooth cracks or kinked cracks can be recovered in the limit $\epsilon\rightarrow 0$ as a  product of Griffith's energy release rate multiplied by the surface area (or its generalization given by  $d-1$ dimensional Hausdorff measure) see Section \ref{Flatcrack}. 

The geometry of the failure set is controlled by how it grows dynamically. Growth is governed by the rate of work done against boundary forces and the dynamic interaction between elastic displacement  and bond failure. Although interaction is captured implicitly through the evolution equations,
one can now apply the rate form of energy balance to explicitly deliver the time rate of the damage energy and characterize the location of the region undergoing damage. This region is called the process zone $PZ^\epsilon(t)$ and from the constitutive law, corresponds to the regions of highest strain. The damage rate and process zone are  determined by  the displacement  field through the rate of work done by the load and the change in both the kinetic energy and elastic potential energy of the specimen. The rate form of energy balance also dictates the onset of crack nucleation and propagation. The rate form of balance and its ramifications are introduced in Section \ref{energybalanceprocesszone}. For a flat mode-I crack in a plate the strain is greatest in a neighborhood of  the tips and this the location of the process zone. For this case the energy balance is in terms of Griffith fracture energy, see Section \ref{Flatcrack}.
On the other hand, away from damaging zones, it is shown that the model delivers the energy  density  associated with isotropic linear elasticity.  Explicit formulas for the Lam\'e constants in terms of the force potentials are obtained, see Section \ref{quescint}. In this way, it is seen  that the energy for this model is given by  the surface energy over failure zones and a volume energy associated with linear elastic behavior inside quiescent zones.  

In Section \ref{vanish} we consider a sequence of nonlocal initial value problems for a crack propagating from left to right, parameterized by $\epsilon$, and pass to the limit of vanishing horizon.  It is shown that the limit displacement field is a solution of the linear elastic wave equation outside a propagating traction free crack. Moreover, the same elastic constants derived in Section \ref{quescint}  appear in the limit of vanishing horizon providing self-consistency for the model. 

We depart from earlier bond breaking approaches of  \cite{EmmrichPhust} and \cite{DuTaoTian}  and add a $\epsilon^{-(d+1)}$ scaling to the bond force and insert a factor of $\sqrt{|\yy-\xx|}$  into the bond force and  bond breaking criteria, see Section \ref{sec:nonlocal:model}.   This change produces the desired model that is linear elastic in quiescent regions away from the damage set but nonlinear elastic with larger strain in the neighborhood of bond breaking. Strikingly, the scaling allows the failure energy to be nonzero and  strictly positive for flat  $d-1$ dimensional failure sets corresponding to creation of internal boundaries i.e., cracks.
Previous work addresses the same constitutive law treated in this paper within a discrete quasistatic model  \cite{BhattacharyaLipton}. Earlier work treats the same strain versus force law but without the presence of a damage factor, that approach is done in the context of tension loading \cite{lipton2014dynamic,lipton2016cohesive,liptonjha2021,lipton2019complex}.  

We note all results are presented in terms of the non-dimensional model. In Section \ref{nond} we rescale and present the associated model in units having the dimensions $M=$ Mass, $L=$ Length, and $T=$ Time.  This enables calibration of the nonlocal model with dimensions to properties of a material with given Lam\'e moduli,  critical energy release rate, and strength. In this paper we focus on central force models to understand scaling necessary for energy balance. Future work is directed toward non-central force models as embodied in the state based peridynamic models \cite{silling2007peridynamic}.

We close the paper with numerical simulations to illustrate crack propagation featuring the alignment of broken bonds  behind the propagating strain concentration. Simulation clearly shows maximum strain energy dissipation \cite{rozen2020fast} as a crack path selection mechanism for a bifurcating crack, see Section \ref{sec:simulations}.


\section{A new formulation}
\label{sec:nonlocal:model}

In what follows the equations of dynamics are expressed in dimensionally reduced form. All quantities including force, time, displacement, and  location of points inside the domain are dimensionally reduced.  The relationship between the dynamic equation in its dimensional and dimensionally reduced  form is presented in Section  \ref{nond}.
The body containing the damaging material $\Omega$ is a bounded domain in two or three dimensions. Nonlocal interactions between a point in the body $\xx$ and its neighbors $\yy$ are confined to the sphere (disk) of radius $\epsilon$ denoted by $H_\epsilon(\xx)=\{\yy:\,|\yy-\xx|<\epsilon\}$. Here $V^\epsilon_d=\omega_d\epsilon^d$ is  the $d$ dimensional volume of the ball $H_\epsilon(\xx)$ centered at $\xx$ where $\omega_d$ is the volume of the $d$ dimensional unit ball.  The elastic displacement $\uu(t,\xx)$ is defined for $0\leq t\leq T$ and $\xx$ in $\Omega$. We write $\uu(t)=\uu(t,\cdot)$ and  introduce the 
two point strain $S(\yy, \xx, \uu(t))$ between the point $\xx$ and any point $\yy \in H_\epsilon(\xx)$  given by
\begin{align}\label{strain}
    S(\yy, \xx, \uu(t)) = \frac{(\uu(t,\yy)-\uu(t,\xx))}{|\yy-\xx|}\cdot\ee, & \hbox{  where  } \ee=\frac{\yy-\xx}{|\yy-\xx|},
\end{align}	
and we set $$r:=r(t,\uu)=\sqrt{|\yy-\xx|}S(\yy,\xx,\uu(t)).$$  The strain  satisfies $ S(\yy, \xx, \uu(t)) = S(\xx, \yy, \uu(t)) $.
The scaled nonlocal kernel is introduced and is given by
\begin{equation}\label{scale}
\rho^\epsilon(\yy,\xx)=\frac{{\chi_\Omega(\yy)}J^\epsilon(|\yy-\xx|)}{\epsilon V^\epsilon_d},
\end{equation}
where {$\chi_\Omega$} is the characteristic function of $\Omega$, $J^\epsilon(|\yy-\xx|)$ is the influence function, a positive function on the  ball of radius $\epsilon$ centered at $\xx$ and is radially decreasing taking the value $M$ at the center of the ball  and $0$  for $|\yy-\xx|\geq\epsilon$. The radially symmetric influence function is written  $J^\epsilon(|\yy-\xx|)=J(|\yy-\xx|/\epsilon)$. The nonlocal kernel is scaled by $\epsilon^{-1}(V_d^{\epsilon})^{-1}$ enabling the model to be characterized by the nondimensional linear elastic shear $\mu$ and Lam\'e $\lambda$ moduli  in regions away from the crack, critical energy release rate $\mathcal{G}^\epsilon_c$, and bond strength  $g'(r_c)$. The introduction of dimensions and calibration to a given material with known shear $\overline{\mu}$ and Lam\'e $\overline{\lambda}$ moduli,  critical energy release rate $\overline{\mathcal{G}_c}$ and strength is established in Section \ref{nond}.

\begin{figure}
   \centering
\includesvg[width=0.8\linewidth]{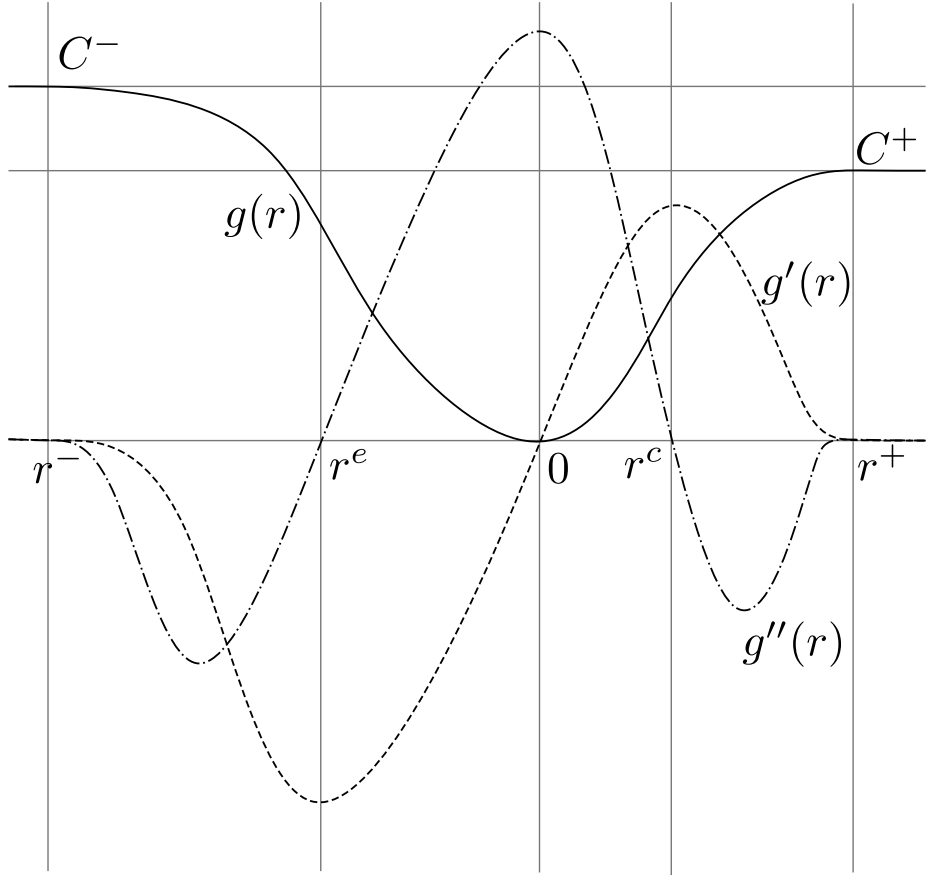}
		  \caption{The potential function $g(r)$ and derivatives $g'(r)$ and $g''(r)$  for tensile force. Here $C^+=g(r^+)$ and $C^- = g(r^-)$ are the asymptotic values of $g$. The derivative of the force potential goes smoothly to zero at $r^+$ and $r^-$.}
		  \label{ConvexConcavea}
\end{figure}

Bond force is related to strain similar to the cohesive force law given in \cite{lipton2014dynamic,lipton2016cohesive}. Under this law, the force is linear for small positive strains and for larger positive strains the force begins to soften and  becomes zero
after reaching a critical strain.   For negative strain the bond force resists  the strain.
These features are encoded into the force function expressed in terms of the derivative of the potential function $g(r)$ defined for,  $-\infty< r <\infty$, with $g(r^-):=C^-$ for $r\leq r^-$, and $g(r^+):=C^+$ for $r^+\leq r$. We choose $r^c$ and $r^e$ such that $r^-<r^c<0<r^e<r^+$ and the potential is convex for  $r^e\leq r\leq r^c$ and concave otherwise. 
Here we can choose $-r^e$ arbitrairily larger than $r^c$. 
The profile of force potential used here is shown  in Figure \ref{ConvexConcavea}. Here, the functions $g',\,g''$   (depicted in Figure  \ref{ConvexConcavea}) satisfy
\begin{align}
\label{prime}
\max_{-\infty<r< \infty}\{| g'(r)|\} <\infty&
\, \hbox{  and }\max_{-\infty<r<\infty}\{| g''(r) |\}<\infty.
\end{align}
The profile is chosen to have bounded  asymptote  $C^-$ to insure the dynamics is well-posed, see Section \ref{wellposedness}.

To construct the constitutive law relating force to strain, set $r^\pm=\sqrt{|\yy-\xx|}S^\pm$, so that  $g'(r)=0$ for $r\leq r^-$ and $r^+\leq r$ and set  $r^c=\sqrt{|\yy-\xx|}S^c$ and   $r^e=\sqrt{|\yy-\xx|}S^e$. We require $g''(r)$ to be continuous on $-\infty<r <\infty$.  The constitutive law relating force to strain and subsequent bond failure is given by
\begin{equation}\label{const1}
\boldsymbol{f}^\epsilon(t,\yy,\xx,\uu)=
\frac{2\rho^\epsilon(\yy,\xx)}{\sqrt{|\yy-\xx|}}\gamma(\uu)(\yy,\xx,t)g'(r(t,\uu))\boldsymbol{e},
\end{equation}
where  $\gamma(\uu)(\yy,\xx,t)$ is the degradation factor that decreases from one when critical strain is reached. The  first form of degradation factor is a scaled version of the one presented  in \cite{EmmrichPhust}.
Let $(r-r^+)^+$ be the positive part of $r-r^+$ and the degradation factor  $\gamma(\uu)(\yy,\xx,t)$ is given by
\begin{align}\label{dfs}
\gamma(\uu)(\yy,\xx,t)=h\left(\int_0^t\,(\,r(\tau,\uu))-r^+)^+\,d\tau\right),
\end{align}
$h=h(x)$ is nonnegative, takes the value one for  arguments $x<0$, decreases  smoothly to zero for $0\leq x<x^D$, and is zero for $x^D\leq x$. 
It is clear from this definition that irrevocable degradation occurs when the strain exceeds $S^+=r^+/\sqrt{|\yy-\xx|}$ and  the time spent above the threshold depends on the strain rate.  The bond between $\xx$ and $\yy$ fails (breaks) when $\gamma(\uu)(\yy,\xx,t)=0$. The time to failure can be made small by taking $x^D$ small.

Set $r^D>r^+$ and
a  second form  of degradation factor is given by the scaled version of the one presented in \cite{DuTaoTian} given by
\begin{align}\label{dfs2}
\gamma(\uu)(\yy,\xx,t)=h \left(\sqrt{|\yy-\xx|}S^*(t,\yy,\xx,\uu)\right),
\end{align}
where $h(r)$ is $1$ for $-\infty \leq r\leq r^+$ and decreases smoothly to zero at $r=r^D$. The maximum strain up to time is $t$ given by
\begin{align*}
S^*(t,\yy,\xx,\uu)=\max_{0\leq \tau\leq t}\{ S(\yy, \xx, \uu(\tau))\}.
\end{align*}
As before, irrevocable degradation occurs when $\sqrt{|\yy-\xx|}S^*(t,\yy,\xx)$ exceeds $r^+$, but now the bond breaks when  $ \sqrt{|\xx-\yy|}S^*(t,\yy,\xx,\uu)\geq r^D>r^+$. Here $r^D-r^+>0$ can be taken arbitrarily small.
In both cases, the degradation factor is symmetric, i.e., $\gamma(\uu)(\yy,\xx,t)=\gamma(\uu)(\xx,\yy,t)$. 
The nonlocal force density $\LL^\epsilon[\uu](t,\xx)$ defined for all points  $\xx$ in $\Omega$  is given by
\begin{align}
      \label{eq: force}
 \LL^\epsilon [\uu](t,\xx) = -\int_{\Omega} \boldsymbol{f}^\epsilon(t,\yy,\xx,\uu) \,d\yy.
\end{align}
The material is assumed to be homogeneous and the balance of linear momentum for each point $\xx$ in the body $\Omega$ in dimensionally reduced form is given by
\begin{align}\label{eq: linearmomentumbal}
\rho\ddot{\uu}(t,\xx)+ \LL^\epsilon [\uu](t,\xx) =\bb(t,\xx),
\end{align}
where  $\bb(t,\xx)$  is a prescribed dimensionally reduced body force density and  $\rho$
is a nondimensional ratio and is given by  \eqref{rhonondim}. 
The  linear momentum balance is supplemented with the initial conditions on the displacement and velocity given by
\begin{align}
\uu(0,\xx)=\uu_0(\xx),\nonumber\\
\dot{\uu}(0,\xx)=\vv_0(\xx),\label{initialconditions}
\end{align}
and we look for a solution $\uu(t,\xx)$ on a time interval $0<t<T$.
This completes the problem formulation.

We observe that nonlocal dynamics for $\uu(t)$  over  $\Omega$ in $\mathbb{R}^d$ presents as dynamics in $\Omega\times(\Omega\cap H_\epsilon(\xx))\subset\mathbb{R}^d\times\mathbb{R}^d$ for the two point strain $S(\yy,\xx,\uu(t))$.  The strain dynamics generates  a set where force is related to strain  and a failure set. The set  of pairs $(\xx,\yy)$ in  $\Omega\times(\Omega\cap H_\epsilon(\xx))$ for which the bond force is related to strain $S(\yy,\xx,\uu(t))$ is the set $EZ^\epsilon(t)$. The failure set $\Gamma^\epsilon(t)$ is the collection of pairs $(\xx,\yy)$ in $\Omega\times(\Omega\cap H_\epsilon(\xx))$ such that  the bond between them was broken at some time $\tau$, $0<\tau\leq t$. 
The failure set is written 
\begin{align}\label{damageset}
\Gamma^\epsilon(t)=\{(\xx,\yy)\hbox{ in }\Omega\times(\Omega\cap H_\epsilon(\xx)); \, \gamma(\uu)(\yy,\xx,t)=0\}
\end{align}
with indicator function
\begin{equation}\label{busted}
\chi^+(\yy,\xx,t)=\left\{\begin{array}{ll}1,&\hbox{if  $(\xx,\yy)$ in $\Gamma^\epsilon(t)$}\\
0,& \hbox{otherwise}.\end{array}\right.
\end{equation}

\section{Wellposedness }
\label{wellposedness}

Body forces are easily chosen for which one can find a unique solution $\uu(t,\xx)$ of the initial value problem  \eqref{eq: linearmomentumbal} and \eqref{initialconditions}.
To this end, we introduce the subspace of $L^\infty(\Omega,\mathbb{R}^d)$ given by all rigid body motions $\mathcal{U}$ is defined by
\begin{equation}\label{rigidmotion}
\mathcal{U}=\{\ww:\,\ww=\mathbb{Q}\xx+\mathbf{c};\, \mathbb{Q} \in \mathbb{R}^{d\times d},\,\,\mathbb{Q}^T=-\mathbb{Q};\,\,\mathbf{c}\in\mathbb{R}^d\}. \end{equation}
Calculation shows that the strain  $S(\yy,\xx,\ww)=0$ for $\ww\in\mathcal{U}$. Moreover, the work of \cite{DuGunLehZho} shows that this is precisely the null space of the strain operator on $L^\infty(\Omega;\mathbb{R}^d)$.
We guarantee solvability requiring that the body force density $\bb(\xx,t)$ satisfy $\int_\Omega\,\ww\cdot\bb,d\xx=0$ for all $\ww\in\mathcal{U}$.
With this in mind, we introduce  the subspace  of $L^\infty(\Omega;\mathbb{R}^d)$ denoted by
$\dot L^\infty(\Omega;\mathbb{R}^d)$
defined to be all elements of $L^\infty(\Omega;\mathbb{R}^d)$ orthogonal to $\mathcal{U}$. 
The space of functions that are twice continuous in time taking values in $X=\dot{L}^\infty(\Omega,\mathbb{R}^d)$ is denoted by $C^2([0,T];X)$. The norm on this space is given by
\begin{align}\label{norms}
\Vert \uu(t,\xx)\Vert_{C^2([0,T]; X)}=\sup_{0\leq t\leq T}\{ \sum_{i=0}^2||\partial_t^i\uu(t,\xx)||_{X}\}.
\end{align}
To simplify notation,  the space $C([0,T];X)$ is written as $CX$ and the  corresponding norm is denoted by $\Vert\cdot \Vert_{CX}$. Similarly, we write $C^2([0,T];X)$ as $C^2X$.
Existence and uniqueness of the solution to the initial boundary value problem is stated below:

\begin{theorem}[{\bf Existence and Uniqueness of Solution}]
\label{exist}
 The initial value problem  given by \eqref{eq: linearmomentumbal} and \eqref{initialconditions} with initial data in $X$ and $\bb(t,\xx)$ belonging to $CX$
has a unique solution $\uu(t,\xx)$ in $C^2X$ with two strong  derivatives in time.
 \end{theorem}

\noindent The initial value problem delivers a displacement damage set pair: $\{\uu^\epsilon(t)$, $\Gamma^\epsilon(t)\}$.

The proof of Theorem \ref{exist} follows \cite{EmmrichPhust} and has new features due to the definition of the bond force potential.
We now show
\begin{enumerate}
\item The operator $ \LL^\epsilon [\uu](t,\xx)$ is a map from $CX$ into itself.
\item The operator $ \LL^\epsilon [\uu](t,\xx)$ is Lipschitz continuous in $\uu$ with respect to the norm of $CX$.
\end{enumerate}
The theorem then follows from an application of the Banach fixed point theorem.

To establish properties (1) and (2), we first summarize the differentiability and Lipschitz continuity of the damage factor. 
 \vspace*{1\baselineskip}
 
 \begin{lemma}[{\bf Differentiablity and Lipshitz Continuity of Damage Factor}]
\label{HCont}
 For $\uu\in CX$ the mapping
\begin{align}\label{properties1}
(\yy,\xx)\mapsto \gamma(\uu)(\yy,\xx,t):\Omega\times \Omega\rightarrow\mathbb{R}
\end{align}
is measurable for every $t\in[0,T]$, and for the degradation function given in \eqref{dfs} the mapping
\begin{align}\label{properties 2}
t\mapsto \gamma(\uu)(\yy,\xx,t):[0,T]\rightarrow\mathbb{R},
\end{align}
is differentiable for almost all $(\yy,\xx)$. Moreover, 
for almost all $(\yy,\xx)\in (\Omega\times \Omega\cap H_\epsilon(\xx))$ and all $t\in [0,T]$, the map
\begin{align}\label{properties 3}
\uu\mapsto \gamma(\uu)(\yy,\xx,t):CX\rightarrow\mathbb{R}
\end{align}
is Lipschitz continuous and for \eqref{dfs}
\begin{align}\label{properties 4}
| \gamma(\uu)(\yy,\xx,t)-\gamma(\ww)(\yy,\xx,t)|\leq \Vert \uu-\ww\Vert_{CX} \frac{tC}{\sqrt{|\yy-\xx|}}
\end{align}
while for \eqref{dfs2}
\begin{align}\label{properties 5}
| \gamma(\uu)(\yy,\xx,t)-\gamma(\ww)(\yy,\xx,t)|\leq \Vert \uu-\ww\Vert_{CX} \frac{C}{\sqrt{|\yy-\xx|}}.
\end{align}
 \end{lemma}
 \begin{proof}
The first two  claims are immediate and \eqref{properties 4} and  \eqref{properties 5} follow as in \cite{EmmrichPhust} and \cite{DuTaoTian} but the factor $(|\yy-\xx|)^{-1/2}$ appears in the denominator of the right hand side  of \eqref{properties 4} and \eqref{properties 5}.\end{proof}
To complete the demonstration of (1)  and (2) we point out 
the key features given below.
 \vspace*{1\baselineskip}
 
\begin{lemma}[{\bf Boundness and Lipshitz continuity of $g'$}]
\label{linftcnttheta} 
Given two functions $\uu$ and $\ww$  in $CX$
\begin{align}\label{lipshitzg}
| g'(\sqrt{|\yy-\xx|}S(\yy,\xx,\uu(t)))-g'\sqrt{|\yy-\xx|}S(\yy,\xx,\ww(t)))|\leq \Vert \uu-\ww\Vert_{CX}\frac{C}{\sqrt{|\yy-\xx|}},
\end{align}
and
\begin{align}\label{lipshitze}
| g'(\sqrt{|\yy-\xx|}S(\yy,\xx,\ww(t)))\ee|\leq C.
\end{align}
 \end{lemma}
 \begin{proof}
The inequatity \eqref{lipshitzg} follows from \cite{bhattaLiptonMMS} and \eqref{lipshitze} follows from the definition of $g'(r)$.
We now establish (2) and recover (1) as a consequence.
Given $\uu$ and $\ww$ in $CX$ 
\begin{align}\label{lipshitzL}
| \LL^\epsilon [\uu](t,\xx) - \LL^\epsilon [\ww](t,\xx)|\leq I+II,
\end{align}
where 
\begin{align}\label{l}
I=\int_{\Omega}| 
\frac{2\rho^\epsilon(\yy,\xx)}{\sqrt{|\yy-\xx|}}\gamma(\uu)(\yy,\xx,t)g'(r(t,\uu))\boldsymbol{e}-
\frac{2\rho^\epsilon(\yy,\xx)}{\sqrt{|\yy-\xx|}}\gamma(\uu)(\yy,\xx,t)g'(r(t,\ww))\boldsymbol{e}| d\yy,
\end{align}
\begin{align}\label{Il}
II=\int_{\Omega}|
\frac{2\rho^\epsilon(\yy,\xx)}{\sqrt{|\yy-\xx|}}\gamma(\uu)(\yy,\xx,t)g'(r(t,\ww))\boldsymbol{e}-
\frac{2\rho^\epsilon(\yy,\xx)}{\sqrt{|\yy-\xx|}}\gamma(\ww)(\yy,\xx,t)g'(r(t,\ww))\boldsymbol{e}| d\yy,
\end{align}

From \eqref{lipshitzg} 
\begin{align}\label{lconclude}
I\leq
C\Vert \vv-\ww\Vert_{CX}\int_{\Omega}\frac{\rho^\epsilon(\yy,\xx)}{{|\yy-\xx|}}\,d\yy,
\end{align}
from \eqref{properties 4} 
\begin{align}\label{Ilconclude}
II\leq
tC\Vert \vv-\ww\Vert_{CX}\int_{\Omega}\frac{\rho^\epsilon(\yy,\xx)}{{|\yy-\xx|}}\,d\yy.
\end{align}

All integrals are bounded for $d=2,3$, so for $0<t<T$ there is a Lipschitz constant $L$ 
such that
\begin{align}\label{LipschitzLL}
\Vert \LL^\epsilon [\uu] - \LL^\epsilon [\ww]\Vert_{(CL^{\infty})^d}\leq L\Vert \uu-\ww\Vert_{CX}.
\end{align}
\end{proof}

\begin{proof}[Proof of Theorem \ref{exist}]
Now one can easily  show the solution $\uu$ is the unique  fixed point of $\uu(t)=(I \uu)(t)$ where $I$ maps $CX$ into itself and is defined by
\begin{equation}\label{Fixedpointmap}
\begin{aligned}
(I\uu)(t)=&\uu_0+t\vv_0+\rho^{-1}\int_0^t(t-\tau)\LL^\epsilon[\uu](\tau)+\bb(\tau)\,d\tau.
\end{aligned}
\end{equation}
This problem is equivalent to finding the unique solution of the initial value problem given by
\eqref{eq: linearmomentumbal} and \eqref{initialconditions}.  We absorb the factor $\rho^{-1}$ into $\mathcal{L}^\epsilon[\uu]+\bb$ and show that $I$ is a contraction map. To see that $I$ is a contraction map an equivalent norm is introduced
\begin{equation}\label{Equivnorm}
\begin{aligned}
|||\uu|||_{CX}=\max_{t\in[0,T_0]}\{e^{-2LTt}\Vert \uu\Vert_{X}\},
\end{aligned}
\end{equation}
For $t\in [0,T]$ one has
\begin{equation}\label{Contrction}
\begin{aligned}
\Vert (I\uu)(t)-(I\ww)(t)\Vert_{X}\leq&\int_0^t(t-\tau)\Vert \LL[\uu](\tau)-\LL[\ww](\tau)\Vert_{X}\,d\tau\\
&\leq LT\int_0^t\Vert \uu-\ww\Vert_{C([0,\tau];X)}\,d\tau\\
&\leq LT\int_0^t\max_{s\in[0,\tau]}\{\Vert \uu(s)-\ww(s)\Vert_{X}e^{-2LTs}\}e^{2LT\tau}\,d\tau\\
&\leq\frac{e^{2LTt}-1}{2}|||\uu-\ww|||_{CX},
\end{aligned}
\end{equation}
hence
\begin{equation}\label{ContrctionD}
\begin{aligned}
|||(I\uu)(t)-(I\ww)(t)|||_{CX}\leq\frac{1}{2}|||\uu-\ww|||_{CX},
\end{aligned}
\end{equation}
and  $I$ is a contraction.  From the Banach fixed point  theorem there is a unique fixed point $\uu(t)$ belonging to $CX$ and it is evident from \eqref{Fixedpointmap} that $\uu(t)$ also belongs to $C^2X$.
\end{proof}

\section{Energy Balance, Bounded Damage Energy, Failure Energy and Process Zone}
\label{energybalanceprocesszone}
The energy balance and damage energy are seen to follow directly from the evolution equation with no hypotheses. The explicit formula for the damage energy in terms of the damage factor and $C^+$ is obtained.  We show these energies are bounded. We then describe the failure zone where bonds have failed and identify the process zone on which bonds are degrading.   

Energy balance in rate form is found by multiplying the evolution equation \eqref{eq: linearmomentumbal} by $\dot\uu(t)$ and integrating over the domain to get

\begin{align}
      \label{eq: linearmomentumbalmult}
\partial_t\left(\int_\Omega\, \rho \, \frac{|\dot{\uu}(t)|^2}{2}\,d\xx\right)+ \int_\Omega\,\LL^\epsilon [\uu](t)\cdot\dot\uu(t)\,d\xx =\int_\Omega\,\bb(t)\cdot\dot\uu(t)\,d\xx,
\end{align}
Writing out the second term
\begin{align}
      \label{eq: 2nd term}
\int_\Omega\,\LL^\epsilon [\uu](t)\cdot\dot\uu(t)\,d\xx =-\int_\Omega\int_\Omega\,\frac{2\rho^\epsilon(\yy,\xx)}{\sqrt{|\yy-\xx|}}\gamma(\uu)(\yy,\xx,t)g'(r(t,\uu))\boldsymbol{e}\,d\yy\cdot\dot\uu(\xx)\,d\xx,
\end{align}
and integration by parts gives 
\begin{align}\label{eq:3nd term}
\int_\Omega\,\LL^\epsilon [\uu](t)\cdot\dot\uu(t)\,d\xx =\int_\Omega\int_\Omega\,\rho^\epsilon(\yy,\xx)\gamma(\uu)(\yy,\xx,t)\partial_t g(r(t,\uu))\,d\yy\,d\xx,
\end{align}
and we get
\begin{align}\label{eq:prebalance}
\partial_t\left(\int_\Omega\, \rho\,\frac{|\dot{\uu}(t)|^2}{2}\,d\xx \right)+\int_\Omega\int_\Omega\,\rho^\epsilon(\yy,\xx)\gamma(\uu)(\yy,\xx,t)\partial_t g(r(t,\uu))\,d\yy\,d\xx =\int_\Omega\,\bb(t)\cdot\dot\uu(t)\,d\xx.
\end{align}
The kinetic energy at time $t$ is given by
\begin{equation}\label{kineticenergy}
\mathcal{K}(t)=\int_\Omega\, \rho\,\frac{|\dot{\uu}(t)|^2}{2}\,d\xx,
\end{equation}
and the elastic energy at time $t$ is given by
\begin{align}\label{elasticenergy}
\mathcal{E}^\epsilon(t)&=\int_\Omega\int_\Omega\,\rho^\epsilon(\yy,\xx)\gamma(\uu)(\yy,\xx,t)g(r(t,\uu))\,d\yy\,d\xx.
\end{align}

The time derivative for the degradation factor defined by \eqref{dfs} is seen to exist for all $\uu$ in $CX$. On the other hand, the time derivative for the  degradation factor defined by \eqref{dfs2} exists for the solution of  the initial boundary value problem \eqref{eq: linearmomentumbal} and \eqref{initialconditions}. To see this note the solution belongs to $C^1X$ so $\gamma(\uu)(\yy,\xx,t)$ is {Lipschitz} continuous in time. The damage energy is
\begin{equation}\label{damageenergy}
\mathcal{D}^\epsilon(t)=-\int_0^t\int_\Omega\int_\Omega\,\rho^\epsilon(\yy,\xx)\partial_t\gamma(\uu)(\yy,\xx,\tau)g(r(\tau,\uu))\,d\yy\,d\xx\,d\tau.
\end{equation}

Note $$\gamma(\uu)(\yy,\xx,t) \partial_tg(r(t,\uu)=\partial_t\{\gamma(\uu)(\yy,\xx,t) g(r(t,\uu)\}-\partial_t\gamma(\uu)(\yy,\xx,t) g(r(t,\uu)$$ 
and we get the rate form of energy balance:
 \vspace*{1\baselineskip}

\noindent{\bf Rate form of Energy Balance}
\begin{align}\label{eq:rateenergybalance}
\partial_t\mathcal{K}(t) + \partial_t\mathcal{E}^\epsilon(t) + \partial_t \mathcal{D}^\epsilon(t)=\int_\Omega\,\bb(t)\cdot\dot\uu(t)\,d\xx.
\end{align}

\noindent The explicit formula for $\mathcal{D}^\epsilon(t)$ is given by:
 \vspace*{1\baselineskip}

\noindent{\bf Damage Energy}
\label{damag}
\begin{align}\label{explicitdamageenergy}
\mathcal{D}^\epsilon(t)=\int_\Omega\int_\Omega\,\rho^\epsilon(\yy,\xx)C^+(1-\gamma(\uu)(\yy,\xx,t))\,d\yy\,d\xx.
\end{align}
The $d-1$ dimensional damage energy follows immediately from the definition of $\rho^\epsilon(\yy,\xx)$ given by \eqref{scale} which delivers the damage energy as an integral  against a $d-1$ dimensional measure. 

To obtain the  explicit formula \eqref{explicitdamageenergy}  one exchanges time and space integrals in \eqref{damageenergy} and  evaluates
\begin{equation}\label{product}
-\int_0^t\partial_t\gamma(\uu)(\yy,\xx,\tau)g(r(\tau,\uu))\,d\tau,
\end{equation}
for  fixed $(\xx,\yy)$. Define $t^I_{\yy,\xx,\uu}$ as the first time that $S(\yy,\xx,\uu(t))>S^+$ and $t^F_{\yy,\xx,\uu}$ as the time that $\gamma(\uu)(\yy,\xx,t)=0$ so $\gamma(\uu)(\yy,\xx,t^I_{\yy,\xx,\uu})=1$ and $\gamma(\uu)(\yy,\xx,t^F_{\yy,\xx,\uu})=0$.
Observe that the desired formula  for \eqref{product} now easily follows noting that  $g(r(t,\uu))=C^+$ for  $S(\yy,\xx,\uu(t))\geq S^+$ and  from  the definition of the degradation factor \eqref{dfs}, 
\begin{equation}\label{derivitiveofdegredation}
-\partial_t\gamma(\uu)(\yy,\xx,t)=\left\{\begin{array}{ll}0,& S(\yy,\xx,\uu(t)\leq S^+\\
>0,& S^+<S(\yy,\xx,\uu(t))\\0, & t^F_{\yy,\xx,\uu}\leq t.\end{array}\right.
\end{equation}
Consequently, the damage energy associated with the bond $(\xx,\yy)$ at time $t$ is given by
\begin{equation}\label{productexplicit}
-\int_0^t\partial_t\gamma(\uu)(\yy,\xx,\tau)g(r(\tau,\uu))\,d\tau=C^+(1-\gamma(\uu)(\yy,\xx,t)),
\end{equation}
and the explicit formula follows. We set $S^D=r^D/\sqrt{|\yy-\xx|}$ and the same explicit formula follows for the degradation factor \eqref{dfs2} noting that
\begin{equation}\label{derivitiveofdegredation2}
-\partial_t\gamma(\uu)(\yy,\xx,t)=\left\{\begin{array}{ll}0,& S^*(t,\yy,\xx,\uu)\leq S^+\\
\geq 0,& S^+<S^*(t,\yy,\xx,\uu)<S^D\\0, & t^F_{\yy,\xx,\uu}\leq t\end{array}\right.
\end{equation}

The failure energy $\mathcal{F}^\epsilon({t})$ is the total damage energy expended to fail all bonds up to time ${t}$ and from  \eqref{explicitdamageenergy} and \eqref{busted} is  given by the $d-1$ dimensional integral

 \vspace*{1\baselineskip}

\noindent{\bf Failure Energy}
\label{failure}
\begin{align}\label{expliciitfailureenergy}
\mathcal{F}^\epsilon(t)=\int_\Omega\int_\Omega\,\rho^\epsilon(\yy,\xx)C^+\chi^+(\yy,\xx,t)\,d\yy\,d\xx.
\end{align}

The evolution  delivering the displacement-damage pair has  bounded elastic, potential, and damage energy given by
 \vspace*{1\baselineskip}

\begin{theorem}[{\bf Energy Bound}]
\label{energybound}
\begin{equation}\label{boundedenergy}
\max_{0<t<T}\left\{ \mathcal{K}(t)+\mathcal{E}^\epsilon(t)+\mathcal{D}^\epsilon(t)\right\}< C.
\end{equation}
Where the constant $ C$ only depends on the initial conditions and the load history.
\end{theorem}

\begin{proof} To find this bound,  write
\begin{align}\label{W}
W(t)=\mathcal{K}(t) + \mathcal{E}^\epsilon(t) + \mathcal{D}^\epsilon(t)+1
\end{align}
and note the rate form of energy balance gives
\begin{align}\label{eq:ineqlity}
\partial_t W(t)=\int_\Omega\,\bb(t)\cdot\dot\uu(t)\,d\xx\leq \Vert\bb(t)\Vert_{L^2(\Omega,\mathbb{R}^d)}\Vert\dot\uu(t)\Vert_{L^2(\Omega,\mathbb{R}^d)}\leq\frac{2}{\rho}\sqrt{W(t)}\Vert\dot\uu(t)\Vert_{L^2(\Omega,\mathbb{R}^d)}.
\end{align}
Integrating the inequality and applying initial conditions gives the desired result
\begin{align}\label{eq:uniformbd}
\mathcal{K}(t) + \mathcal{E}^\epsilon(t) + \mathcal{D}^\epsilon(t)\leq \left(\frac{1}{\sqrt{\rho}}\int_0^t\Vert\bb(\tau)\Vert_{L^2(\Omega,\mathbb{R}^d)}\,d\tau+\sqrt{W(0)}\right)^2-1.
\end{align}
\end{proof}
 Applying the  Rate form of Energy Balance and collecting results shows that the damage energy only changes when  $-\partial_t\gamma(\uu)(\yy,\xx,t)>0$ and is determined by the evolving displacement field, through the change in elastic energy, kinetic energy and work done against the load.  These observations are summarized in the following Lemma.

\begin{lemma}[{\bf  Growth of the Damage Energy and Process Zone}]
\label{damagezone}
The set of pairs $(\xx,\yy)$ for which $-\partial_t\gamma(\uu)(\yy,\xx,t)>0$ is called the process zone  $PZ^\epsilon(t)$  and $\partial_t\mathcal{D}^\epsilon(t)>0$  only over the process zone we have the power balance:
\begin{align}\label{eq:processzonepowerbalance}
\partial_t\mathcal{D}^\epsilon(t)=\int_\Omega\int_\Omega\,\rho^\epsilon(\yy,\xx)C^+(-\partial_t\gamma(\uu)(\yy,\xx,t))\,d\yy\,d\xx =\int_\Omega\,\bb(t)\cdot\dot\uu(t)\,d\xx-\partial_t\mathcal{K}(t) - \partial_t\mathcal{E}^\epsilon(t)  > 0.
\end{align}
\end{lemma}
It is observed that conditions for which damage occurs follows directly from \eqref{eq:processzonepowerbalance} and is given below.
\begin{lemma}[{\bf  Condition for Damage}]
\label{damagecondition}
If the rate of energy put into the system exceeds the material's capacity to generate kinetic and elastic energy through  displacement and velocity, then damage occurs.
\end{lemma}

In summary, the rate form of energy balance  shows that $\partial_t\mathcal{D}^\epsilon(t)>0$ only on  $PZ^\epsilon(t)$ and zero elsewhere. It is seen that the rate $\partial_t\mathcal{D}^\epsilon(t)$ and location of $PZ^\epsilon$ is given in terms of  the displacement field,  through the rate of work done by the load and the change in both the kinetic energy and elastic potential energy of the specimen. If a fissure exists, the strain is greatest in the neighborhood of its tips and this is the location of the process zone.

The energy in the process zone is denoted by $\mathcal{P}^\epsilon(t)=\int_{PZ^\epsilon(t)}\rho^\epsilon(\yy,\xx)C^+(1-\gamma(\uu)(\yy,\xx,t))\,d\yy\,d\xx$  and we integrate the rate form of energy balance \eqref{eq:rateenergybalance} and state the energy balance

\begin{lemma}[{\bf  Energy balance}]
\begin{align}\label{eq:enegrynergybalance}
 \mathcal{P}^\epsilon(t)+ \mathcal{F}^\epsilon(t) =\int_0^t\int_\Omega\,\bb(\tau)\cdot\dot\uu(\tau)\,d\xx\,d\tau-(\mathcal{K}(t) + \mathcal{E}^\epsilon(t) -(\mathcal{K}(0) + \mathcal{E}^\epsilon(0) +\mathcal{D}^\epsilon(0))).
\end{align}
\end{lemma}

Lastly, we see that the set of pairs $(\yy,\xx)$ corresponding to bonds not irrevocably broken is given by the set
\begin{align}\label{elastic}
EZ^\epsilon(t)=\{(\xx,\yy)\in\,\Omega\cup\Omega\cap H_\epsilon(\xx)\setminus  \Gamma^\epsilon(t)\}
\end{align}
and $PZ^\epsilon(t) \subset EZ^\epsilon(t)$.

In summary,  we observe that nonlocal dynamics over  $\Omega$ in $\mathbb{R}^d$ generates an evolution over a subset of $\mathbb{R}^d\times\mathbb{R}^d$ for the two point strain $S(\yy,\xx,\uu(t))$.  This delivers a displacement-failure set pair and energy balance.

\section{Failure Energy as Geometric Integral, Flat Cracks, Griffith Fracture Energy, and Energy Balance}
\label{Flatcrack}
In this section, the formula for the $d-1$ dimensional integral defining the failure energy is derived. We then explicitly compute the critical energy release rate per unit area (length)  necessary to generate a unit of fracture surface for this model. We find that the failure energy for a flat $d-1$ failure domain to discover it is given by the product of energy release rate multiplied by the area (length) of the surface (line segment) across which bonds are broken.   We discover that the failure energy to be nonzero and  strictly positive for flat  $d-1$ dimensional failure sets corresponding to creation of internal boundaries i.e., cracks.  Last when strain is large enough to make the force-strain relation unstable and bonds break we recover energy balance in terms of Griffith fracture energy.

We change variables  $\boldsymbol{\zeta}=(\yy-\xx)/\epsilon$ and switch the order of integration in \eqref{expliciitfailureenergy} to get the $d-1$ dimensional integral
\begin{align}
\label{1stenergy}
\mathcal{F}^\epsilon({t})=\frac{C^+}{\omega_d}\int_{H_1(0)}\,J(|\bzz|)\frac{1}{\epsilon}\left(\int_\Omega\chi^+(\xx+\epsilon\bzz,\xx,{t})\,d\xx\right)\,d\bzz.
\end{align}
We change coordinates in $\xx$ using the slicing variables and write $x=\zz+\ee s$ Where $s\in\mathbb{R}$,  $\ee=\bzz/|\bzz|$ is the direction along $\bzz$ and $\zz$ is the $d-1$ dimensional subspace of points $\zz\perp\ee$. For $(\zz,\ee)$ fixed, consider the line $\Omega_\zz^\ee=\{s\in\mathbb{R};\,\zz+\ee s\}$ and $ \Omega^e=\{\zz;\,\Omega_\zz^\ee\cap\Omega\not=\emptyset\}$ and 
\begin{align}
\label{2ndtenergy}
\mathcal{F}^\epsilon({t})=\frac{C^+}{\omega_d}\int_{H_1(0)}\,J(|\bzz|)\frac{1}{\epsilon}\int_{\Omega^\ee}\int_{\Omega^\ee_\zz}\chi^+(\zz+\ee(s+\epsilon|\bzz|),\zz+\ee s,{t})\,ds\,d\zz\,d\bzz.
\end{align}
We change to polar coordinates and $d\bzz=|\bzz|^{d-1}d|\bzz|\,d\ee$, where $d\ee$ is the surface area element of $\mathbb{S}^{d-1}$. After a change in the order of integration we arrive at the desired formula.
 \vspace*{1\baselineskip}

\noindent{\bf  Geometric integral representation for Failure energy}
\label{geometricintegral}

 \vspace*{1\baselineskip}
\begin{align}
\label{3rdtenergy}
\mathcal{F}^\epsilon({t})=\frac{C^+}{\omega_d}\int_{\mathbb{S}^{d-1}}\int_{\Omega^\ee}\int_0^1\,J(|\bzz|)|\bzz|^d m^\epsilon({t},\ee,\zz,|\bzz|)\,d|\bzz|\,d\zz\,d\ee,
\end{align}
with
\begin{align}
\label{1stkernel}
m^\epsilon({t},\ee,\zz,|\bzz|)=\frac{1}{\epsilon|\bzz|}\int_{\Omega^\ee_\zz}\chi^+(\zz+\ee(s+\epsilon|\bzz|),\zz+\ee s,{t})\,ds,
\end{align}
where
\begin{align}
\label{kernelkernel}
\chi^+(\zz+\ee(s+\epsilon|\bzz|),\zz+\ee s,{t})=\left\{\begin{array}{ll}1,&\hbox{if the pair $(\zz+\ee(s+\epsilon|\bzz|),\zz+\ee s)$ is in $\Gamma^\epsilon({t})$}\\
0,& \hbox{otherwise}.\end{array}\right.
\end{align}
The function $m^\epsilon({t},\ee,\zz,|\bzz|)$ is associated with the intersection of the line $\xx=\zz+\ee s$ with  the subset of bonds of length $|\yy-\xx|$ in $\Gamma^\epsilon(t)$ divided by the length of the bond $|\yy-\xx|$.

Next the energy per unit area to make new surface for a crack is calculated. 
Referring to \eqref {elasticenergy} the elastic energy density  is given by
\begin{align}\label{energydensity}
W^\epsilon(\xx,\uu)=\int_\Omega\,\rho^\epsilon(\yy,\xx)\gamma(\uu)((\yy,\xx,t))g(r(t,\uu))\,d\yy.
\end{align}
We are interested in the work necessary to take undamaged material and fail it irrevocably. The energy density is also the stress power density and is implicated in inelastic processes when stresses get sufficiently large.
The portion of stress power density in the upper half plane $\mathcal{K}^+$, see Figure \ref{compute} is
\begin{align}\label{energydensityup}
W_+^\epsilon(\xx,\uu)=\int_{\Omega\cap\mathcal{K}^+}\,\rho^\epsilon(\yy,\xx)g(r(t,\uu))\,d\yy.
\end{align}
From  \eqref{explicitdamageenergy} and \eqref{productexplicit} the energy necessary to break a bond of length $|\yy-\xx|$  is given by $J(|\yy-\xx|)C+/\epsilon$. This is consistent with failure corresponding to maximum energy dissipation of bond energy.
 The fracture toughness $\mathcal{G}_c^\epsilon$ is defined to be  the energy per unit area required to send the force between each point $\xx$ and $\yy$ on either side of a planar surface to zero. Because of the finite length scale of interaction only the force between pairs of points within an $\epsilon$ distance from the surface need to be considered and collecting results gives
 \begin{eqnarray}
\mathcal{G}_c^\epsilon={2}\int_0^\epsilon \,W_+^\epsilon(x_1,z,x_3,\uu)\,dz.
\label{calibrate3formula}
\end{eqnarray}
 For $d=3$ the  fracture toughness $\mathcal{G}_c^\epsilon$ is calculated as in \cite{lipton2016cohesive,jhalipton2020} and given by the formula
\begin{eqnarray}
\mathcal{G}_c^\epsilon=\frac{2}{\omega_3}\int_0^\epsilon\int_0^{2\pi}\int_z^\epsilon\int_0^{\arccos(z/\zeta)}J(\zeta)\frac{C^+}{\epsilon}\zeta^2\sin{\psi}\,d\psi\,d\zeta\,d\theta\,dz
\label{calibrate2formula}
\end{eqnarray}
where $\zeta=|\yy-\xx|/\epsilon$, see Figure \ref{compute}. A similar computation can be carried out for two dimensional problems. Calculation delivers the  formulas for $d=2,3$ given by
\begin{eqnarray}
\mathcal{G}_c^\epsilon=\mathcal{G}_c=C^+\frac{2\omega_{d-1}}{\omega_d}\, \int_{0}^1 r^dJ(r)dr, \hbox{  for $d=2,3$}
\label{epsilonfracttough}
\end{eqnarray}
where $\omega_{d}$ is the volume of the $n$ dimensional unit ball, $\omega_1=2,\omega_2=\pi,\omega_3=4\pi/3$.  This shows that the critical energy release rate is independent of $\epsilon$ for this model.
\begin{figure} 
\centering
\begin{tikzpicture}[xscale=0.78,yscale=0.78]
\draw [-,dashed, thick] (0,3.5) -- (0,-1);
\draw [->,ultra thick] (0,.95) -- (0,0);
\draw [->,thick] (0,0) -- (2.8,.95);
\draw[->,thick] (0,0) -- (-1.7,1.7);
\draw [-,thick] (-3.2,.95) -- (3.2,.95);
\node [right] at (-1.7,1.7) {${\bf y}$};
\node [left] at (0,.55) {${z}$};
\draw [ultra thick] (2.7,1) arc [radius=3, start angle=25, end angle=155];
\node [right] at (0,-0.2) {${\bf x}$};
\draw [<-](1.05,1.425) arc (45:90:1.5);
\draw [-] (0,0) -- (1.24,1.7);
\node [below] at (1.24,1.7) {$\zeta$};
\draw [->] (.9,3.5) arc (0:315:.85cm and .3cm);
\node [below] at (2.5,.75) {$\epsilon$};
\node [right] at (2.5,1.8) {$\mathcal{K}^+$};
\node [right] at (.8,3.3) {$\theta$};
\node [right] at (.0,2.0) {$\scriptscriptstyle{\arccos(z/\zeta)}$};
\end{tikzpicture} 
\caption{{\bf Evaluation of fracture toughness $\mathcal{G}_c^\epsilon$. For each point $\xx$ along the dashed line, $0\leq z\leq \epsilon$, the work required to break the interaction between $\xx$ and $\yy$ in the spherical cap is summed up in \eqref{calibrate2formula} using spherical coordinates centered at $\xx$.}}
 \label{compute}
\end{figure}
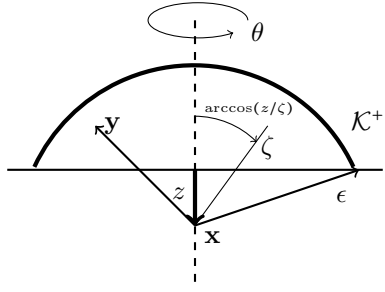


As an example consider the failure set $\Gamma^\epsilon(t)$ defined by a flat $d-1$ dimensional piece of surface (line segment) $R_t$  where points  above the surface are no longer influenced by forces due to points below the surface and vice versa. This  is the case of alignment, i.e., all bonds connecting points $\yy$  above  $R$ to points $\xx$ below are broken and vice versa. Let $\hat{s}$ be the $s$ coordinate of a line $\xx=\zz+\ee s$ piercing the planar surface (line segment) $R_t$. For this case, a straightforward calculation gives 
\begin{align}
\label{evalkernelkernel}
\chi^+(\zz+\ee(s+\epsilon|\bzz|),\zz+\ee s,{t})=\left\{\begin{array}{ll}1,&\hbox{for  $s$ in the closed interval $[\hat{s}-\epsilon|\bzz|,\hat{s}]$}\\
0,& \hbox{otherwise}.\end{array}\right.
\end{align}
so
\begin{align}
\label{ml}
m^\epsilon({t},\ee,\zz,|\bzz|)=\left\{\begin{array}{ll}1,&\hbox{if the line $\xx=\zz+\ee s$ pierces $R_t$}\\
0,& \hbox{otherwise}.\end{array}\right.
\end{align}
for all $|\bzz|<1$ and we write $m^\epsilon({t},\ee,\zz,|\bzz|)=N(t,\ee,\zz)$, where $N(t,\ee,\zz)$ is the multiplicity function of any line $\Omega_\zz^\ee$ counting the number of times it pierces the surface (line segment). For this case, $R_t$ is flat so $N(t,\ee,\zz)$ is either one or zero. The failure energy becomes
\begin{align}
\label{crofton}
\mathcal{F}^\epsilon({t})=\mathcal{G}_c\frac{1}{2\omega_{d-1}}\int_{\mathbb{S}^{d-1}}\int_{\Omega^\ee}N(t,\ee,\zz)\,d\zz\,d\ee.
\end{align}
The second factor is Crofton's formula and delivers the surface area of the internal boundary $R_t$, or more generally, the $d-1$ dimensional Hausdorff measure of $R_t$ written $\mathcal{H}^{d-1}(R_t)$ \cite{Morgan,Federer}, and
\begin{align}
\label{last}
\mathcal{F}^\epsilon({t})=\mathcal{G}_c\mathcal{H}^{d-1}(R_t).
\end{align}
This is the well known formula for Griffith fracture energy but now derived directly from a nonlocal peridynamic model. 
 What is distinctive is  that  the Griffith fracture energy  found here  follows directly from the nonlocal model without sending any parameter such as  $\epsilon$ to zero as in other approaches to free fracture. Of course we can extend the formula \eqref{last} to a system of dispersed flat cracks separated by the distance $\epsilon$ with different orientations.

 More generally Crofton's formula delivers the $d-1$ Hausdorff measure of any $d-1$ dimensional set that is rectifiable \cite{Morgan,Federer}.  With this in mind  consider the failure set $\Gamma^\epsilon(t)$ now defined by a smooth rectifiable  $d-1$ dimensional piece of surface $R_t$  where again points  above the surface are no longer influenced by forces due to points below the surface and vice versa. We have again  that \eqref{evalkernelkernel} holds where $\hat{s}$ lies on the surface. However if two or more points on the line are also on the the surface $R_t$ and are less than $\epsilon$ apart we have  
 \begin{align}\label{boundonm}
m^\epsilon({t},\ee,\zz,|\bzz|)\leq N(t,\ee,\zz).
\end{align}
Additionally if we fix $R_t$ and send $\epsilon\rightarrow 0$  
\begin{align}\label{limitofm}
\lim_{\epsilon\rightarrow 0}{m^\epsilon({t},\ee,\zz,|\bzz|)}=N(t,\ee,\zz).
\end{align}
 So we conclude from  Lebegue's convergence theorem that
 \begin{align}\label{limF}
\lim_{\epsilon\rightarrow 0}\mathcal{F}^\epsilon({t})=\mathcal{G}_c\frac{1}{2\omega_{d-1}}\int_{\mathbb{S}^{d-1}}\int_{\Omega^\ee}N(t,\ee,\zz)\,d\zz\,d\ee=\mathcal{G}_c\mathcal{H}^{d-1}(R_t).
\end{align}
From the definition of rectifiability \cite{Federer},  if a crack set $R_t$  is  rectifiable  there exist countably many compact subsets $K_i $ of $C^1$ graphs such that
\begin{align}\label{defpfrect}
\mathcal{H}^{d-1}\left( R_t\setminus \cup_{i\geq 1}K_i \right)=0,
\end{align}
and we see that the formula for the failure energy recovers a classic Griffith fracture energy generalized to rectifiable fracture sets given by
\begin{align}\label{limFrect}
\lim_{\epsilon\rightarrow 0}\mathcal{F}^\epsilon({t})=\mathcal{G}_c\mathcal{H}^{d-1}(R_t).
\end{align}
This is consistent with the surface energy  of fracture  recovered from the nonlocal model in the local limit for  the dynamic case where cracks can heal  \cite{lipton2014dynamic,lipton2016cohesive}.

The examples show how the failure energy can recover the Griffith fracture energy for flat cracks with $\epsilon>0$ and its generalization to countably rectifiable cracks but only for $\epsilon\rightarrow 0$. However crack shape is determined by how the failure set grows in time. Because of this the failure sets or cracks emerge from the dynamic bond breaking process. Growth of the failure set is determined by loading history and the shape of the sample.


\section{Linear elastic energy density in quiescent regions}
\label{quescint}
In this section it is shown that the model delivers a volume energy associated with strains away from damaging zones, i.e., $\gamma(\uu)(\yy,\xx,t)=1$.  Here, we show the model recovers the linear elastic energy density away from the damaging region.  
As before, the energy density is given by \eqref{energydensity}, i.e.,
\begin{align}\label{energydensity2}
W^\epsilon(\xx,\uu)=\int_\Omega\,\rho^\epsilon(\yy,\xx)\gamma(\uu)((\yy,\xx,t))g(r(t,\uu)\,d\yy.
\end{align}
For small strains  and  $\gamma(\uu)(\yy,\xx,t)=1$, it is shown  below that the energy density is a volume density  related to the strain energy density of a linearly elastic material. For this case, the energy density is described to leading order by shear and Lam\'e moduli $\mu$ and $\lambda$. To illustrate the ideas, suppose the strain field is smooth across ${H}_\epsilon(\xx)$ and Taylor expansion gives  $\frac{1}{2}(\nabla\uu(t,\xx)+\nabla\uu(t,\xx)^T)\ee\cdot\ee \approx S(\yy,\xx,\uu(t))$.   Let $g(r)=f(r^2)$ where $f$ is strictly increasing concave and $f(0)=0$, $f'(r^+) =0$ and  $f(r^+)=g_\infty$. Set  $F=(\nabla\uu(t,\xx)+\nabla\uu(t,\xx)^T)/2$ and  the straight forward calculation outlined  below reveals that $\mu$ and $\lambda$ describe the strain energy density to leading order for ${S}=F\ee\cdot \ee<{S}^c=r^c/\sqrt{|\yy-\xx|}$, i.e.,
\begin{eqnarray}
W^\epsilon(\xx,\uu)=\int_\Omega\,\rho^\epsilon(\yy,\xx)g(r(t,\uu))\,d\yy=2\mu |F|^2+\lambda |Tr\{F\}|^2+O(\epsilon|F|^4),
\label{LEFMequality}
\end{eqnarray}
where
\begin{eqnarray}\label{lambdamu}
\mu=\lambda=\frac{1}{8} g''(0)\int_{0}^1r^2J(r)dr, \hbox{ $d=2$}&\hbox{ and }& \mu=\lambda=\frac{1}{10} g''(0)\int_{0}^1r^3J(r)dr, \hbox{ $d=3$}.\label{calibrate1}
\end{eqnarray}
To see this, it is assumed that ${S}<<{S}^c$ and we expand $g(r(t,\uu))$ in a Taylor series about $0$ 
 noting that $g'(0)=g'''(0)=0$ to get
\begin{align}
&g(r(t,\uu))=\frac{g''(0)}{2!}|\yy-\xx|{S}^2+\frac{g''''(0)}{4!}(|\yy-\xx|{S}^2)^2+\cdots,\hbox{ and}\nonumber\\
&\frac{1}{\epsilon}J^\epsilon(|\yy-\xx|)g(r(t,\uu))=\frac{1}{\epsilon}J^\epsilon(|\yy-\xx|)(\frac{g''(0)}{2!}|\yy-\xx|S^2+\frac{g''''(0)}{4!}(|\yy-\xx|{S}^2)^2+\cdots).
\label{taylorf}
\end{align}
Substitution of \eqref{taylorf} into \eqref{energydensity2} with $\gamma(\uu)(\yy,\xx,t)=1$ and the change of variables $\xi=(\yy-\xx)/\epsilon$  delivers the energy density  restricted to quiescent regions given by
\begin{eqnarray}
W^\epsilon(\xx,\uu)=\frac{g''(0)}{2\omega_d}\int_{H_1(0)}|\xi|J(|\xi|)(F\ee\cdot \ee)^2\,d\xi+O(\epsilon|F|^4),
\label{densityyexpto2nd}
\end{eqnarray}
where ${H}_1(0)$ is the unit ball centered at the origin and $\omega_d$ is its volume $d=2,3$.
Observe next that $(F\ee\cdot \ee)^2=\sum_{ijkl}F_{ij}F_{kl}e_ie_je_ke_l$ and the leading order term in \eqref{densityyexpto2nd} is given by 
\begin{eqnarray}
\sum_{ijkl}\mathbb{C}_{ijkl}F_{ij}F_{kl}
\label{leading order}
\end{eqnarray}
where
\begin{eqnarray}
\mathbb{C}_{ijkl}=\frac{g''(0)}{2\omega_d}\int_{\mathcal{H}_1(0)}|\xi|J(|\xi|)\,e_i e_j e_k e_l\,d\xi
\label{elasticpart3}
\end{eqnarray}
The identity \eqref{LEFMequality} now follows as in \cite{lipton2016cohesive}.

\section{The limit of Vanishing Nonlocality in Spatial Variables}
\label{vanish}
We consider an example of a flat crack propagating  from left to right in a plate. In this section, we pass to the limit of vanishing horizon to find that the limit displacement field is a weak solution of the linear elastic
wave equation outside a propagating traction free crack. The elastic coefficients are given precisely by  \eqref{lambdamu}.
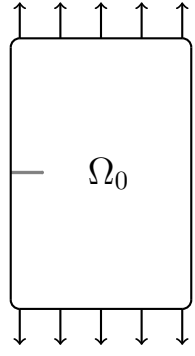
\begin{figure} 
\centering
\begin{tikzpicture}[xscale=0.60,yscale=0.60]

%
%






\draw [fill=gray, gray] (-2,-0.00) rectangle (-1.3,0.05);


\draw[fill=gray, gray] (-1.3,-0.0) arc (-90:90:0.025)  -- cycle;




\draw [-,thick] (-2,2.8) to [out=90, in=180] (-1.8 ,3.0);

\draw [-,thick] (-2,-2.8) to [out=-90, in=180] (-1.8 ,-3.0);

\draw [-,thick] (1.8,3.0) to [out=0, in=90] (2.0 ,2.8);

\draw [-,thick] (1.8,-3.0) to [out=0, in=-90] (2.0 ,-2.8);


%
%
\draw [-,thick] (-2,2.8) -- (-2,-2.8);
\draw [-,thick] (-1.8,-3) -- (1.8,-3);
\draw [-,thick] (2,2.8) -- (2,-2.8);
\draw [-,thick] (1.8,3) -- (-1.8,3);
%
%
%
%
\node [right] at (-0.5,0.0) {{\Large $\Omega_0$}};

%
%

\draw[->,thick] (-1.8,3.0) -- (-1.8,3.80);

\draw[->,thick] (-0.9,3.0) -- (-0.9,3.80);

\draw[->,thick] (0.0,3.0) -- (0.0,3.80);

\draw[->,thick] (0.9,3.0) -- (0.9,3.80);

\draw[->,thick] (1.8,3.0) -- (1.8,3.80);


\draw[->,thick] (-1.8,-3.0) -- (-1.8,-3.80);

\draw[->,thick] (-0.9,-3.0) -- (-0.9,-3.80);

\draw[->,thick] (0.0,-3.0) -- (0.0,-3.80);

\draw[->,thick] (0.9,-3.0) -- (0.9,-3.80);

\draw[->,thick] (1.8,-3.0) -- (1.8,-3.80);

\end{tikzpicture} 
\caption{{ \bf Plate with initial crack on the left edge}}
 \label{singlenotch}
\end{figure}

Define the region $\Omega$ given by a  rectangle with rounded corners, see figure \ref{singlenotch}.  The domain lies within the rectangle $\{0<x_1<a;\,-b/2<x_2<b/2\}$ and a pre-crack is present on the left side of the specimen, see figure \ref{singlenotch}. 
The domain containing the pre-crack is denoted by $\Omega_0$. The pre-crack is described by a line lying on the $x_2=0$ axis given by the interval $0\leq x_1\leq \ell(0)$. The pre-crack can be written as
\begin{equation}\label{centerat0}
R_0=\{0 \leq x_1\leq\ell^\epsilon(0),\,x_2=0\}.  
\end{equation}
Pairs of points with $\xx$ above and $\yy$ below the pre-crack passing between them have pairwise force ${\bolds{f}}^\epsilon(\xx,\yy)$ equal to zero.

The specimen  $\Omega$ pulled apart by an $\epsilon$ thickness layer of body force on the top and bottom of the domain consistent with plain strain loading. In the nonlocal setting the ``traction'' is given by the layer of body force on the top and bottom of the domain. For this case, the body force is written as 
\begin{equation}\label{bodyforce}
\begin{aligned}
\bb^\epsilon(t,\xx)=\be^2\epsilon^{-1}{gt,(x_1,)}\chi_+^\epsilon(x_1,x_2) \hbox{  on the top layer and}\\
\bb^\epsilon(t,\xx)=-\be^2\epsilon^{-1}{g(t,x_1,)}\chi_-^\epsilon(x_1,x_2)\hbox{  on the bottom  layer,}
\end{aligned}
\end{equation}
where $\be^2$ is the unit vector in the vertical direction, $\chi^\epsilon_+$ and $\chi^\epsilon_-$ are the characteristic functions of the boundary layers given by
\begin{equation}\label{layers}
\begin{aligned}
\chi_+^\epsilon(x_1,x_2)&=1\,\,\,\hbox{on $\{\theta<x_1<a-\theta, \,b/2-\epsilon<x_2<b/2\}$ and $0$ otherwise,}\\
\chi_-^\epsilon(x_1,x_2)&=1\,\,\,\hbox{on $\{\theta<x_1<a-\theta, \,-b/2<x_2<-b/2+\epsilon\}$ and $0$ otherwise},
\end{aligned}
\end{equation}
where $\theta$ is the radius of curvature of the rounded corners of $\Omega$.
The top and bottom traction forces are equal and in opposite directions and
$g(x_1,t)>0$. We take the function $g$ to be smooth and bounded in the variables $x_1$ and $t$ and define $\bg$ on $\partial \Omega$
such that
\begin{equation}\label{l2}
\begin{aligned}
\bg(t,x_1)=\pm\be^2g(t,x_1) \hbox{ on } \{\theta\leq x_1\leq a-\theta,\,\,x_2 = \pm b/2\}\text{ and } \bg=0\hbox{ elsewhere on }\partial\Omega.
\end{aligned}
\end{equation}
One checks that $\bb(t,\xx)$ satisfies the solvability condition and is in $CX$ and we obtain a unique solution $\uu^\epsilon$ in $C^2X$ to the initial value problem \eqref{eq: linearmomentumbal} and \eqref{initialconditions}.
In view of the symmetry of the loading and initial domain $\Omega_0$  and if the loading is sufficiently large we expect a symmetric crack to form and grow continuously.  Specifically we suppose a flat crack propagates from left to right for $0<t<T$.  
The flat crack is given by  $R^\epsilon_t=\{0 \leq x_1\leq\ell^\epsilon(t),\,x_2=0\}$ for $t\in(t,T)$.  Pairs of points with $\xx$ above and $\yy$ below the crack passing between them soften and break have pairwise force ${\bolds{f}}^\epsilon(\xx,\yy)$ equal to zero.  We suppose the force is sufficient to grow the crack monotonically, i.e., $\ell^\epsilon(t)<\ell^\epsilon(t')$ for $t<t'$ and pairs of points with $\xx$ above and $\yy$ below the crack passing between them have pairwise force ${\bolds{f}}^\epsilon(\xx,\yy)$ equal to zero.  In this example the crack does not propagate all the way through the sample, i.e., $\ell^\epsilon(T)<a-\delta$, for every $\epsilon$ where $\delta<a$ is a  fixed positive constant. The part of the domain not on the crack is denoted by $\Omega^\epsilon_t$ The crack is portrayed in Figure \ref{P}. The projection of the failure set $\Gamma^\epsilon(t)$ onto $\Omega^\epsilon_t$ is given by the grey zone in Figure \ref{P}
which is the union of all peridynamic neighborhoods $H_\epsilon(\xx)$ such that the Lesbesgue measure of the set $\yy$ such the bond between $\yy$ and $\xx$ is broken. 
The failure energy is 
\begin{align}\label{fenergy}
\mathcal{F}^\epsilon(t)=\mathcal{G}_c\mathcal{H}^{1}(R^\epsilon_{t}).
\end{align} 
The set of bonds for which $r^c\leq r(t,\uu)<r^+$ is called the softening zone $SZ^\epsilon(t)$ it corresponds to reversible strain and is not part of the crack.  The softening zone is portrayed in the computational examples.
Noting that pairs of points with $\xx$ above and $\yy$ below the crack passing between them have bond forces that soften to zero elastically before breaking implies that the elastic energy associated with the domain $\Omega_t^\epsilon$  satisfies
\begin{align}\label{elasticenergy2}
\mathcal{E}^\epsilon(t)&=\int_{{\Omega_t}^\epsilon\times{\Omega_t^\epsilon}\cap{EZ^\epsilon(t)}}\,\rho^\epsilon(\yy,\xx)\gamma(\uu)(\yy,\xx,t)g(r(t,\uu))\,d\yy\,d\xx,\nonumber\\
&=\int_{\Omega}\int_{\Omega}\,\rho^\epsilon(\yy,\xx)g(r(t,\uu))\,d\yy\,d\xx.
\end{align}
Note $\gamma(u)(\yy,\xx,t)=1$ in the last integral as bonds break after they soften.


\begin{figure} 
\centering
\begin{tikzpicture}[xscale=0.70,yscale=0.50]

\draw [fill=gray, gray] (-8,-0.5) rectangle (0.0,0.5);






\node[above] at (-4.0,0.5) {$R_t^\epsilon(t)$};

\draw [->,thick] (-4.0,0.8) -- (-4.0,0.0);

\draw[fill=gray, gray] (0,-0.5) arc (-90:90:0.5)  -- cycle;












\draw [-,thick] (-8,0) -- (-0.0,0);




\draw [thick] (-8,-5) rectangle (8,5);












\node [above] at (-1.6,0) {$\epsilon$};

\node [below] at (-1.6,0) {$\epsilon$};

\draw [thick] (-1.3,.5) -- (-1.1,.5);

\draw [<->,thick] (-1.215,.5) -- (-1.215,0);

\draw [<->,thick] (-1.215,-.5) -- (-1.215,0);

\draw [thick] (-1.3,-.5) -- (-1.1,-.5);

\draw [->,thick] (0.0,-1.0) -- (0.0,-0.1);

\node  at (0.5,-1.5) { $\ell^\epsilon(t)$};

\draw [->,thick] (-7,-1.0) -- (-7,-0.1);

\node  at (-7,-1.5) { $\ell(0)$};

\end{tikzpicture} 
\caption{{ \bf The crack $R_t^\epsilon$ and projection of $\Gamma^\epsilon(t)$ onto $\Omega$ given by grey shaded region.}}
 \label{P}
\end{figure}
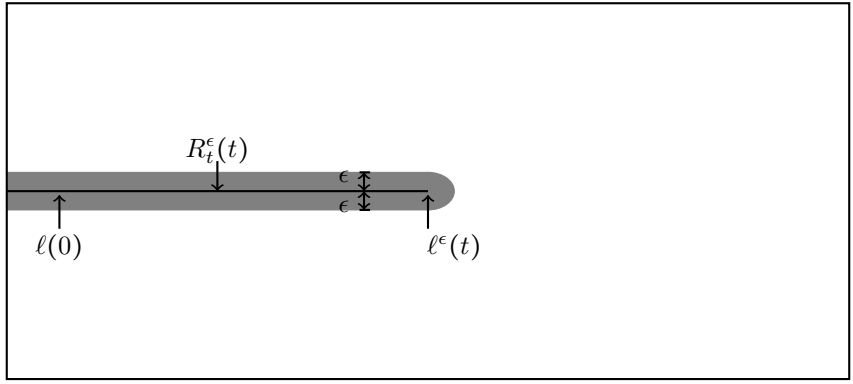

For the rest of the section we make the hypotheses 
\begin{itemize}


\item  (H1) $\sup_{\epsilon>0}\Vert \uu^\epsilon\Vert_{CX}<\infty$

\item (H2) $|SZ^\epsilon(t)|=O(\epsilon^2)$

\end{itemize}
Here $|SZ^\epsilon(t)|$ denotes $d$ dimensional Lebesgue measure of $|SZ^\epsilon|$ and (H2) is consistent with the numerical simulations given in the next section.

Now introduce the Banach space 
\begin{align}\label{Y}
H=\{\uu \in L^2(\Omega,\mathbb{R}^2);\, \int_{\Omega}\,\uu\cdot\vv\,d\xx=0, \hbox{ for all $\vv\in\mathcal{U}$}\}.
\end{align}

The elastic energy is the peridynamic energy $PD^\epsilon(\uu)$ given in  \cite{lipton2016cohesive}, The energy bound \eqref{eq:enegrynergybalance} shows
\begin{align}\label{boundenergpd}
\sup_{\epsilon>0}\{\mathcal{E}^\epsilon(t)\}<\infty.
\end{align}

From  Lemma 6.1 and (6.18) of   \cite{lipton2016cohesive} one has the coercivity
\begin{align}\label{coerse}
\Vert \uu^\epsilon \Vert_H^2 \leq C \mathcal{E}^\epsilon(t)
\end{align}
for some positive constant $C$ so $\uu^\epsilon$ is bounded in $L^2$.
On applying (H1) and \eqref {boundenergpd}  we get hypotheses sufficient for  compactness as stated in 
Theorem 5.1 of \cite{lipton2016cohesive}. This  provides a subsequence of solutions $\{\uu^\epsilon\}$
converging strongly in $C([0,T];H)$ to $\uu^0$ in $C([0,T];H)$.

Applying the Helly selection principle to the sequence $\{\ell^\epsilon(t)\}$ shows there exists a subsequence also denoted by  $\{\ell^\epsilon(t)\}$ converging point-wise  to  a  monotone increasing  continuous and bounded function $\ell^0(t))$ delivering the crack 
$R^0_t=\{0 \leq x_1\leq\ell^0(t),\,x_2=0\}$,  where  $R^0_\tau\subset R^0_t$ for $\tau< t$.  The time dependent domain surrounding the crack is denoted by $\Omega^0_t$ and the failure energy of the limiting crack is
\begin{align}\label{fenergy-0}
\mathcal{F}^0(t)=\mathcal{G}_c\mathcal{H}^{1}(R^0_{t}),
\end{align} 
This gives the limiting displacement-failure pair $(\uu^0,R_t^0)$. 

Next we show that $\uu^0$ is a weak solution to the linear elastic wave equation in $\Omega_t$ with traction free conditions on the faces of the moving crack  $R_t^0$.
To this end recall the Sobolev space $H^1(\Omega;\mathbb{R}^2)$ 
with norm
\begin{equation}\label{h1}
\Vert\ww\Vert_{H^1(\Omega;\mathbb{R}^2)}:=\left(\int_ \Omega\,|\ww|^2\,d\bx+\int_D|\nabla\ww|^2\,d\bx\right)^{1/2}.
\end{equation}
and   $V_t= H^1(\Omega_t;\mathbb{R}^2)\cap H)$.
The Hilbert space dual to $V_t$ is denoted by $V_t^\ast$.  The set of functions strongly square  integrable in time taking values in $V_T^\ast$ for $0\leq t\leq T$ is denoted by $L^2(0,T;V_T^\ast)$.  For future reference we write the symmetric part of $\nabla\ww$ as $\mathcal{E}\ww=(\nabla \ww+\nabla\ww^T)/2$.

The body force given in \eqref{bodyforce} is written as $\bb^{\epsilon}(t)$  and we state the following lemma.
\begin{lemma}\label{convergingrhs}
\cite{liptonjha2021} There is a positive constant $C$ independent of $\epsilon$ and $t\in [0,T]$ such that
\begin{equation}\label{righthandside}
\begin{aligned}
|\langle\bb^{\epsilon}(t),\ww\rangle|
\leq C\Vert\ww\Vert_{V_t},\hbox{ for all $\epsilon>0$ and $\ww \in  V_t$},
\end{aligned}
\end{equation}
where $\langle\cdot,\cdot\rangle$ is the duality paring between $V_t$ and its Hilbert space dual $V^\ast_t$. In addition there exists $\bb^0(t)$ such that $\bb^{\epsilon}\rightharpoonup\bb^0$ in  $L^2(0,T;V_t^\ast)$ and
\begin{equation}\label{righthandsidelimit}
\begin{aligned}
\langle\bb^0(t),\ww\rangle=&\langle\bg(t),\ww\rangle:=\int_{\partial \Omega}\,\bg(t)\cdot\ww\,d\sigma,
\end{aligned}
\end{equation}
for all $\ww \in V_t$, where $\bg(t)$ is defined by \eqref{l2} and $\bg\in H^{-1/2}(\partial\Omega)^2$.
\end{lemma}
The traction force  \eqref{righthandsidelimit}  delivers loading consistent with a mode one crack in the local model given by Linear Elastic Fracture Mechanics (LEFM).

\begin{definition}\label{defweakspaces}
\cite{DalToader} $\mathcal{V}$ is the space of functions $\bv\in  L^2(0,T;\, V_T)\cap H^1(0,T;\,H) $ such that $\bv(t) \in V_t$ for a.e. $t\in (0,T)$. It is a Hilbert space with scalar product given by
\begin{equation}
\label{innerprodspacetime}
(\uu,\vv)_\mathcal{V}=(\uu,\vv)_{ L^2(0,T;\, V_T)}+(\dot\uu,\dot\vv)_{ L^2(0,T;\,H)},
\end{equation}
where $\dot\uu$ and $\dot\vv$ denote distributional derivatives with respect to $t$.
\end{definition}

Following \cite{DalToader} one has
\begin{definition}\label{defweaksoln}
Given $\bg(t)$ defined by \eqref{l2} the displacement $\uu$ is said to be a weak solution of the wave equation on the time changing domain $\Omega_t$ with traction free boundary condition on $R^0_t$ 
\begin{align}\label{formulation}
\begin{cases}
\rho \ddot\uu(t)+{\rm div}(\mathbb{C}\mathcal{E}\uu(t))=0 \\
\mathbb{C}\mathcal{E}\uu(t)\nn=\bg(t), \text{ on }\partial D\\
\mathbb{C}\mathcal{E}\uu(t)\nn=0, \text{ on either side of } R_t^0\\
\bu(t)\in V_t
\end{cases}
\end{align}
on the time interval $[0,T]$ if $\uu\in \mathcal{V}$ and
\begin{equation}
\label{weaksoln}
-\int_0^T\int_D\,\rho\,\dot\uu(t)\cdot\dot\varphi(t)\,d\bx\,dt+\int_0^T\int_D\mathbb{C}\mathcal{E}\uu(t):\mathcal{E}\varphi(t)\,d\xx\,dt=\int_0^T\int_{\partial D}\bg(t)\cdot\varphi(t)\,d\sigma\,dt
\end{equation}
for every $\varphi\in\mathcal{V}$ with $\varphi(T)=\varphi(0)=0$.
\end{definition}

Assuming (H1) and (H2)  as in (2.26) and (3.6) of \cite{liptonjha2021}  delivers:
\begin{theorem}\cite{liptonjha2021} 
\label{zerohorizonweaksoln}
The limit displacement $\uu^0$ is a weak solution of the wave equation on the evolving domain $\Omega_t$ for $t\in[0,T]$ given by  definition \ref{defweaksoln}.
\end{theorem}
\noindent The elasticity tensor $\mathbb{C}$ appearing in Definition \ref{defweaksoln}  is given by \eqref{elasticpart3} with elastic constants given by \eqref{lambdamu}.
A stronger version of the  traction free condition on the sides of the crack is given in  \cite{liptonjha2021}.

\section{Calibration of the model to known material properties}
\label{nond}

We begin this section showing the relationship between the dimensional and dimensionally reduced form of  the equation of dynamics.  We then show how to calibrate the model and conclude with an upper limit on the horizon size for numerical modeling. 

Let $L$ be a characteristic length associated with the domain $\bar{\Omega}=L\Omega$ and, $\bar{\xx}=L\xx$, $\bar{\yy}=L\yy$ and $\bar{r}=L^{1/2}r$. Let the horizon with dimensions of length be given by $\bar{\epsilon}=L\epsilon$.   Let $T_0$ be a characteristic time scale and $\bar{t}=T_0t$ and let $\bar{\rho}$ denote mass per unit volume given by $M/L^d$. The displacement with units of length is $\bar{\uu}=L\uu$. Time derivatives are related to nondimensional time $t$ by $\partial_{\bar{t}}=T_0^{-1}\partial_t$.

We set
\begin{equation}\label{eq:scalllle}
c(\bar{\yy},\bar{\xx},\bar{\epsilon},\epsilon):=\frac{\chi_{\bar{\Omega}}(\bar{\yy})J^{\bar{\epsilon}}(|\bar{\yy}-\bar{\xx}|)}{\epsilon V_d^{\bar{\epsilon}}},
\end{equation}
where $\chi_{{\bar\Omega}}$ is the characteristic function of $\bar\Omega$ and $V_d^{\bar{\epsilon}}=\bar{\epsilon}^d\omega_d$. We emphasize that in the denominator of \eqref{eq:scalllle} we have retained the non dimensional factor $\epsilon$ and note $V_d^{\bar{\epsilon}}=L^d\,V_d^\epsilon$.
The bond potential has dimensions of energy $ML^2T_0^{-2}$ is given by $\bar{g}(\bar{r})=g_\infty g(L^{-1/2}\bar{r})$ where $g$ is the dimensionless bond potential and $g_\infty$ has dimensionss  $ML^2T_0^{-2}$.  Recall, the damage factor $\gamma$ is dimensionless,  $0\leq \gamma\leq 1$ and since  $\bar{r}=L^{1/2}r$ the damage factor  satisfies $\gamma(\bar{\uu})(\bar{\yy},\bar{\xx},\bar{t})=\gamma({\uu})({\yy},{\xx},{t})$.

The relationship between the dimensional and nondimensional form of the linear balance of momentum is now established.
We discuss the $d=3$ case first and comment on the specialization to the plane strain case afterwards.
The nonlocal nondimensional force density defined in Section \ref{sec:nonlocal:model}  is $\LL^\epsilon[\uu](t,\xx)$ and is defined for all points  $\xx$ in $\Omega$  and is given by
\begin{align}
      \label{eq: forccce}
 \LL^\epsilon [\uu](t,\xx) = -\int_{\Omega} \frac{2\rho^\epsilon(\yy,\xx)}{\sqrt{|\yy-\xx|}}\gamma(\uu)(\yy,\xx,t)g'(r(t,\uu))\boldsymbol{e} \,d\yy.
\end{align}

The material is assumed to be homogeneous with nondimensional density $\rho$ and the dimensionally reduced balance of linear momentum for each point $\xx$ in the body $\Omega$ is given by
\begin{align}\label{eq: linearmmmomentumbal}
\rho\ddot{\uu}(t,\xx)+ \LL^\epsilon [\uu](t,\xx) =\bb(t,\xx),
\end{align}
where $\bb(t,\xx)$ is a prescribed dimensionless body force density.

The nonlocal force density $\LL^{\bar{\epsilon}}[\bar{\uu}](\bar{t},\bar{\xx})$ defined for all points  $\bar{\xx}$ in $\bar\Omega$  is given by
\begin{align}
      \label{eq: forcccce}
 \LL^{\bar{\epsilon}}[\bar{\uu}](\bar{t},\bar{\xx}) = -\int_{\bar{\Omega}} \frac{2c(\bar{\yy},\bar{\xx},\bar{\epsilon},\epsilon)}{\sqrt{|\bar{\yy}-\bar{\xx}|}}\gamma(\bar{\uu})(\bar{\yy},\bar{\xx},\bar{t})\bar{g}'(\bar{r}(\bar{t},\bar{\uu}))\boldsymbol{e} \,d\bar{\yy}.
\end{align}
To see that $ \LL^{\bar{\epsilon}}[\bar{\uu}](\bar{t},\bar{\xx}) $ is a force density with units units $MLT^{-2}_0/L^3$,  note that $$\bar{g}'(\bar{r})=\partial_{\bar{r}}\bar{g}(\bar{r})=\partial_{\bar{r}}g_\infty g(L^{-1/2}\bar{r})=L^{-1/2}g_\infty g'(L^{-1/2}\bar{r}),$$
and $\chi_{\bar\Omega}(\bar{\xx})=\chi_{\bar\Omega}(L{\xx})=\chi_{\Omega}({\xx})$,  $J^{\bar{\epsilon}}(|\bar{\yy}-\bar{\xx}|)=J^{{\epsilon}}({|\yy}-{\xx}|)$  hence
\begin{align}\label{ratiotimesforce}
\LL^{\bar\epsilon} [\bar{\uu}](\bar{t},\bar{\xx}) = \frac{g_\infty}{L^{4}}\LL^\epsilon[\uu](t,\xx).
\end{align}
Since $g_\infty/L^4=MLT_0^{-1}/L^3$  we conclude from \eqref{ratiotimesforce} that  the nonlocal force density has units given by $MLT_0^{-2}/L^3$. 
Collecting results shows that the balance of linear momentum for each point $\bar{\xx}$ in the body $\bar{\Omega}$ is given by
\begin{align}\label{eq: linearmmmmmomentumbal}
\bar{\rho}\ddot{\bar{\uu}}(\bar{t},\bar{\xx})+ \LL^{\bar\epsilon} [\bar{\uu}](\bar{t},\bar{\xx}) =\bar{\bb}(\bar{t},\bar{\xx}),
\end{align}
where $\bar{\bb}(\bar{t},\bar{\xx})$ is a prescribed body force density and each term in \eqref{eq: linearmmmmmomentumbal} has the dimensions of force density given by $MLT_0^{-2}/L^3$.  We observe that dividing \eqref{eq: linearmmmmmomentumbal} by $\frac{g_\infty}{L^{4}}$  delivers \eqref{eq: linearmmmomentumbal}. The nondimensional density is given by
\begin{align}\label{rhonondim}
\rho=\frac{\bar{\rho}L^2}{g_\infty T_0^2}.
\end{align}


The $d=2$ case corresponds to plane stress loading in a plate of thickness $\bar{B}$ with cross section $\bar{S}$ and $\bar{\Omega}=\bar{B}\times \bar{S}$.  For this case  $\bar{S}$ has a characteristic length $L$ and we write $\bar{B}=LB$ where $B$ is dimensionless. Each term in \eqref{eq: linearmmmmmomentumbal} corresponds again to $MLT_0^{-2}/L^3$ but all terms are interpreted as force per unit thickness divided by area. 

Material properties given by tabulated data can be used to calibrate the nonlocal model.
Calculation shows that for $d=3$ that
\begin{equation}\label{3dstart}
\bar{\mathcal{G}}_c={g_\infty\mathcal{G}_c}/{L^2}, \qquad \bar{\mu}={g_\infty\mu}/{L^3},
\end{equation}
and for $d=2$,
\begin{equation}\label{2dstart}
\bar{\mathcal{G}}_c={g_\infty\mathcal{G}_c}/{L\bar{B}}, \qquad \bar{\mu}={g_\infty\mu}/{L^2\bar{B}}. 
\end{equation}
The bond strength $g'(r^c)$ is related to the failure strength $\sigma_F$ of the material through
\begin{align}\label{strengthhh}
\sigma_F=\frac{g_\infty g'(r^c)}{L^3}.
\end{align}
We set $g_\infty$ equal to one unit of energy and applying
equations \eqref{3dstart}, \eqref{2dstart}, \eqref{strengthhh},
\eqref{epsilonfracttough}, and \eqref{lambdamu} deliver the system of equations for calibrating $C^+$, $g''(0)$, and $g'(r^c)$  in terms shear moduli,  critical energy release rate and strength of a given material.  

When considering fracture nucleation the horizon $\epsilon$ should lie below the minimum  radius of curvature of an included void or reentrant boundary.   
Here as in~\cite{DiehlLiptonSchweitzer,BhattacharyaLipton}, the strength can be used to determine the maximum horizon size.  On applying the Griffith criterion the horizon $\epsilon$ is constrained above by 
\begin{align}
    \label{eq:griffith}
\epsilon\leq \frac{\bar{\mathcal{G}}_c \bar{E}}{\pi \sigma_F^2 }.
\end{align}


\section{Simulations and Maximum Energy Dissipation}
\label{sec:simulations}

In this section, we carry out numerical simulation of a crack in a rectangular domain (see \Cref{fig:straight-diagram}). 
The critical strain $S^c(\yy,\xx)=r^c/\sqrt{|\yy-\xx|}$ is the strain value for which the bond force between $\yy$ and $\xx$ becomes unstable and the bond force decreases with increasing tensile strain. Neighborhoods $H_\epsilon(\xx)$ with two point strain $S(\yy,\xx,\uu(t))$ below $S^c(\yy,\xx)$ belong to the intact material and bonds for with two point strain is below $S^c(\yy,\xx)$ are in the strength domain of the material \cite{bhattaLiptonMMS}. The elastic strain energy defined for each point inside the intact material is given by
\begin{align}\label{energydensity4}
W^\epsilon(\xx,\uu)=\int_\Omega\,\rho^\epsilon(\yy,\xx)g(r(t,\uu))\,d\yy.
\end{align}
Strain concentration within the intact material is assessed using the maximum value of bond strain relative to the critical strain $S^c(\yy,\xx)$ among all bonds connected to $\xx$ given by 
\begin{align}\label{Z}
   Z(\xx,\uu) = \max_{\yy\in H_\epsilon(\xx)}\left(\frac{r(t,\uu)}{r^c} \right)=\max_{\yy\in H_\epsilon(\xx)}\left(\frac{S(\yy,\xx,\uu(t))}{S^c(\yy,\xx)} \right).
\end{align}
The quantity $Z(\xx,\uu)$ is referred to as the strain concentration. The value of $Z(\xx,\uu)$ is one or less in the intact material.
In this Section we display simulation results for prescribed loading showing crack path, strain concentration $Z(\xx,\uu)$, and strain energy density $W^\epsilon(\xx,\uu)$ at different times. In all simulations we use the degradation factor given by \eqref{dfs2}. To illustrate ideas we assume in the simulations that bonds break instantly and $S^D-S^+$ is set to zero. 
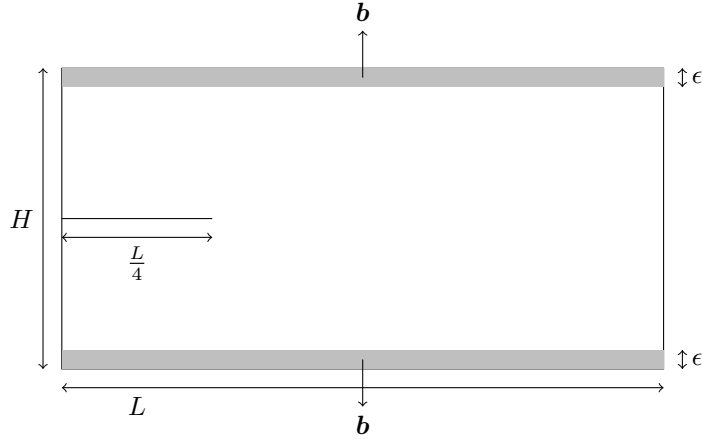
\begin{figure}[tb]
    \centering
    \begin{tikzpicture}
    \draw (0,0) rectangle (8,4);
    \draw (0,2) -- (2,2);
    \fill [lightgray] (0,0) rectangle (8,0.25);
    \fill [lightgray] (0,3.75) rectangle (8,4);
    \draw[<->] (0,-0.25) -- (8,-0.25);
    \node[below] at (1,-0.25) {$L$};
    \draw[<->] (-0.25,0) -- (-0.25,4);
    \node[left] at (-0.25, 2) {$H$};
    
    \draw[<->] (0,1.75) -- (2,1.75);
    \node[below] at (1, 1.75) {$\frac{L}{4}$};
    
    \draw[<->] (8.25,3.75) -- (8.25,4);
    \node[right] at (8.25, 0.125) {$\epsilon$};
    \draw[<->] (8.25,0) -- (8.25,0.25);
    \node[right] at (8.25, 3.875) {$\epsilon$};
    
    \draw[->] (4,0.125) -- (4,-0.5);
    \node[below] at (4, -0.5) {$\bb$};
    \draw[->] (4,3.875) -- (4,4.5);
    \node[above] at (4, 4.5) {$\bb$};
    \end{tikzpicture}
    \caption{Rectangular domain with a horizontal pre-notch}
    \label{fig:straight-diagram}
\end{figure}

\begin{figure}[ht!]
    \centering
    \includegraphics[width=0.99 \linewidth]{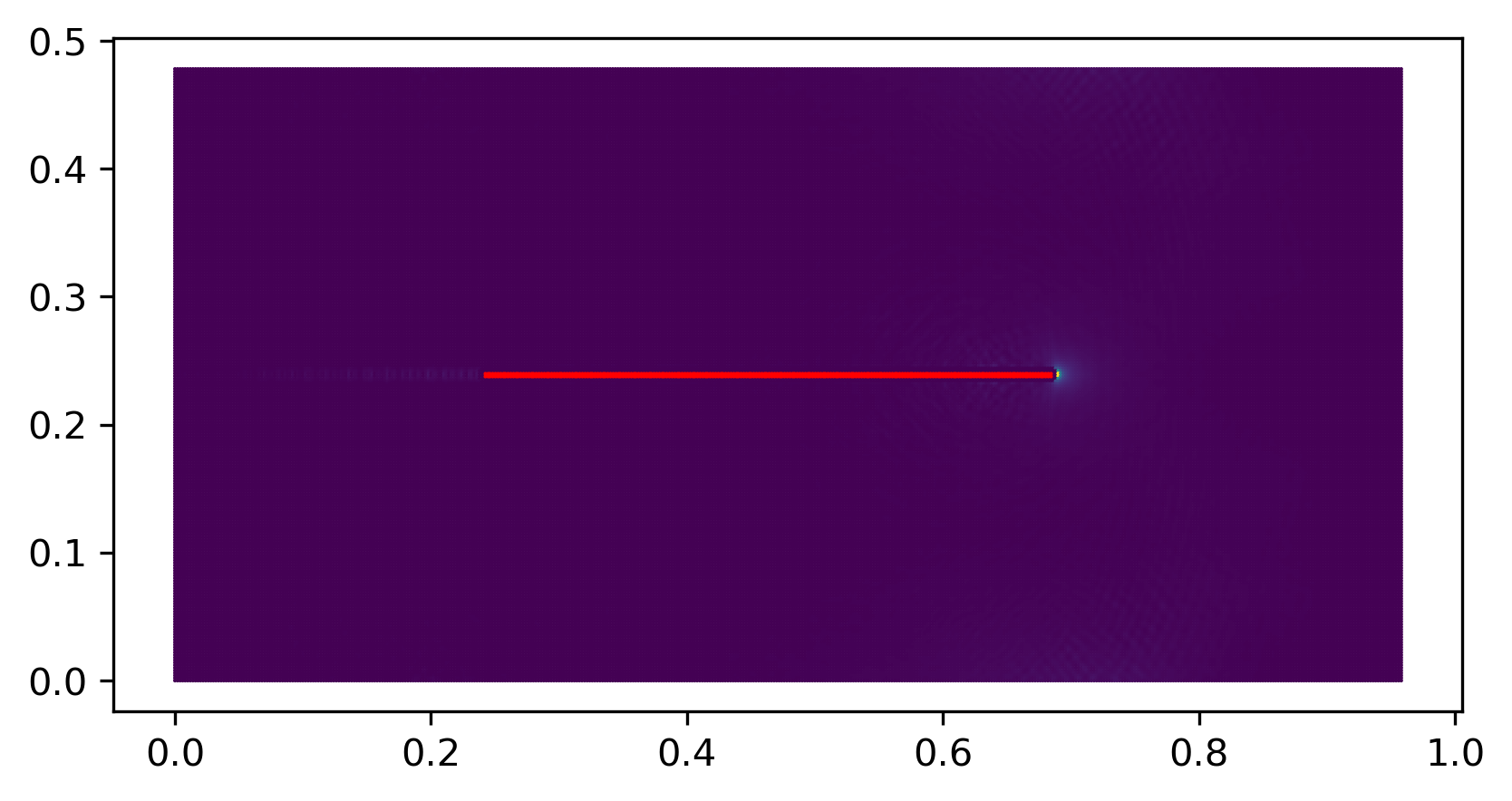}
    \includegraphics[width=0.99 \linewidth]{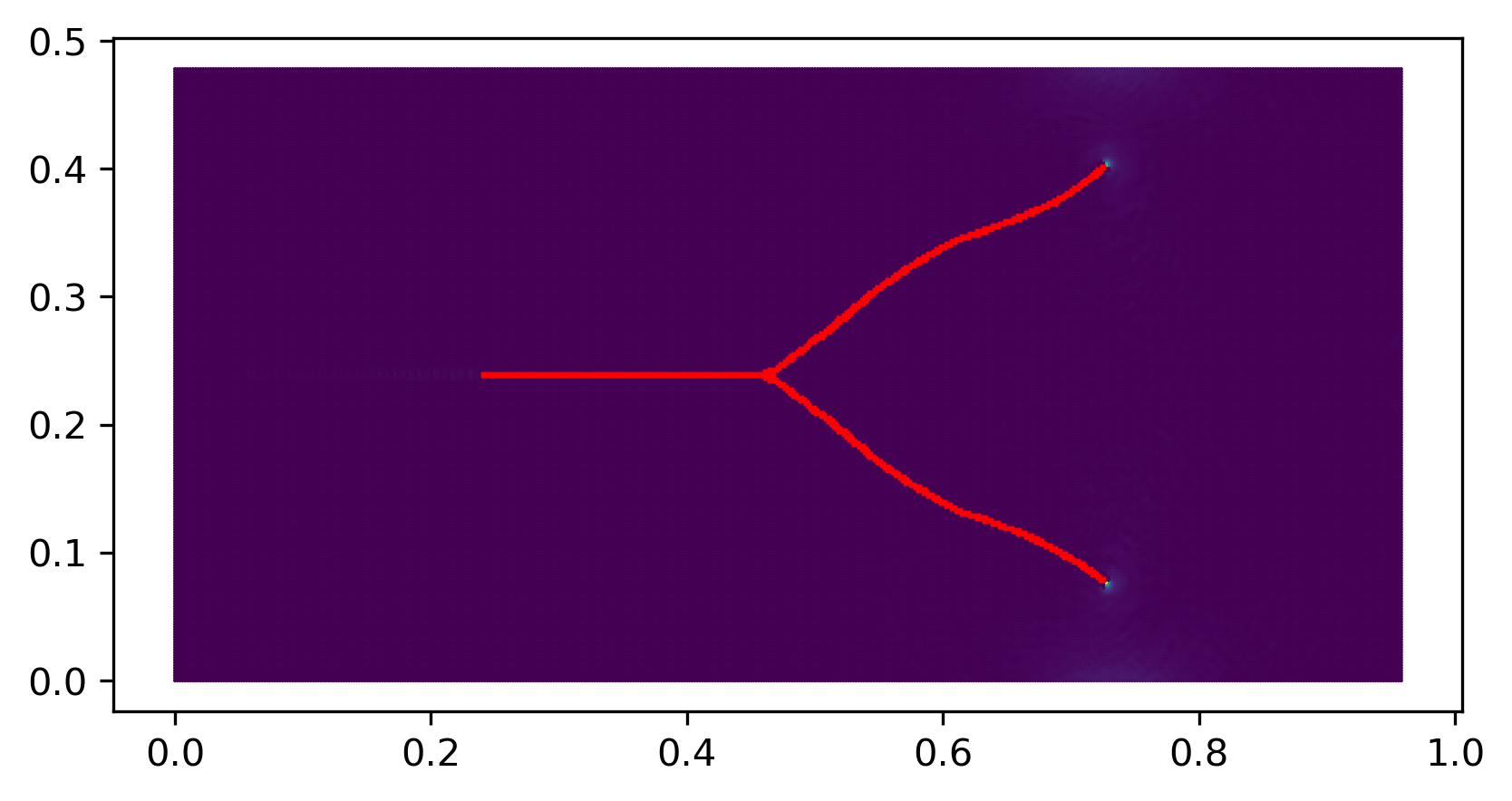}
    \caption{Straight crack and a bifurcated crack due to external body force densities 0.2 GPa and 0.3 GPa, respectively. The red color represents the computed crack.}
    \label{fig:straight-sim}
\end{figure}

Our simulations are carried out for the loading considered in Section \ref{vanish}. For this case, it is shown that cracks  travel along a straight line and for larger loading they bifurcate \cite{ChandarKnauss1984}, \cite{BobaruZhang}. The simulations of Section \ref{setup} reveal that the crack travels along a straight line behind a localized strain concentration as considered theoretically in Section \ref{vanish}. For larger loading the crack is seen to bifurcate. Observational evidence and rational for bifurcation are experimentally and theoretically identified in \cite{ChandarKnauss1984}. We examine strain energy density inside the intact material before and after branching. Before and after the crack branches the crack path is seen to follow
the maximum strain energy of intact material taken over the domain. The simulations explicitly show it to be in front of the crack tip. This is in accord with the experiments of \cite{rozen2020fast} showing that crack path is determined by maximal energy density dissipation.

\subsection{Simulation setup and results}
\label{setup}
We consider a rectangular domain of length $L_0=960$ mm and width   $L_0/2=$ 480 mm with a horizontal pre-notch of length $\frac{L_0}{4}=240$ mm starting at the left edge of the domain (see Fig. \ref{fig:straight-diagram}). 
An external body force density of $\bar{\bb} = 0.2$ GPa and $0.3$ GPa are applied to the top and and bottom $\bar{\epsilon}$-strip of the domain.  The initial displacement and velocity are set to zero.
The domain is discretized uniformly with mesh size $\bar{h}=2$ mm.
The peridynamic horizon size is taken to be $\bar{\epsilon} = 6$ mm.
The material properties are listed in Table \ref{tab:material}.
\begin{table}[]
    \centering
    \begin{tabular}{c|c}
Young's modulus ($E$)     &  72 GPa \\
Critical energy release rate ($G_c$)     & 135 J/m$^2$
\\
Density ($\rho$) & 2440 kg/m$^3$
\\
Poisson ratio ($\nu$) & 0.33
\end{tabular}
    \caption{Material parameters used in simulations}
    \label{tab:material}
\end{table}

The nonlinear potential function considered here is given by
\begin{align*}
    {g}(r) = {C}( 1 - e^{-{\beta} {r}^2})
\end{align*}
where ${C}, {\beta} > 0$.
The influence function is taken to be
\begin{align*}
    J(r) = 1 - r.
\end{align*}
Given the shear modulus $\mu$ and the critical energy release rate $\mathcal{G}_c$, in dimension $d=2$,  \eqref{lambdamu}, \eqref{epsilonfracttough} give
\begin{align}
    {C} = 3 \pi \mathcal{G}_c,\ {\beta} = \frac{16 \mu}{\pi {\mathcal{G}_c}},\ r^c = \frac{1}{\sqrt{2 {\beta}}}.
\end{align}
Note that, the model parameters ${\beta}$ and ${C}$ are independent of the peridynamic horizon size ${\epsilon}$. This model is simple to use and theoretically breaks bonds only when the displacement ceases to be H\"older continuous (with H\"older exponent 1/2). Since $g'(r)$ decays to zero rapidly after $r = r^c$, we choose to break bonds stretched beyond $3r^c$. For this case bonds sustain $5\%$ force relative to the maximum force.

The numerical simulation is performed using the velocity-Verlet scheme with time step size $\Delta t = 5 \times 10^{-7}$ s.
At the start of the simulation, all peridynamic bonds intersecting with the pre-notch are removed.
For a body force density of $0.2$ GPa, the crack propagates in a straight line. For a larger body force density $0.3$ GPa, the crack initially grows in a straight line until it branches at time $t = 535$ $\mu$s.
The branching behavior is consistent with dynamic simulations seen in \cite{ha2010studies}.
The snapshots of the simulations at t = 900 $\mu$s are shown in  \Cref{fig:straight-sim}, where the red color represents the computed crack.
The maximum strain energy density in the domain is observed to be in front of the crack tip, leading the direction of crack growth \cite{rozen2020fast}. 
The vertical displacement of all points in the domain at t = 450 and 700 $\mu$s are shown in \Cref{fig:uy}. Here, the jump set of the displacement field denotes the crack path.
In \Cref{fig:contour-w}, we show
the contour lines of the strain energy density field directly in front of the crack tip. For the straight crack, this is shown at two different crack lengths. For the bifurcating crack, we show the contour lines before and after the crack is bifurcated. The contour lines associated with the stress concentration $Z(\xx,\uu)$ are shown in \Cref{fig:contour-z}.

\section{Conclusion}

The free discontinuity problem for fracture mechanics is investigated  and  a nonlocal model is proposed that  preserves energy balance.
An explicit formula for the energy necessary for material failure is obtained.   
It is observed that conditions for which damage occurs follows directly from \eqref{eq:processzonepowerbalance}, see Lemma \ref{damagecondition}  and is expressed as:
If the rate of energy put into the system exceeds the material's capacity to generate kinetic and elastic energy through  displacement and velocity, then damage occurs.
For flat cracks the failure energy is the classic Griffith energy given by the critical energy release rate and the area of the crack. 
Proceeding theoretically an  $\epsilon$ parametrized family of solutions to the nonlocal boundary value problem  associated with flat crack propagation with $\epsilon>0$ are considered. In the limit of vanishing nonlocality $\epsilon=0$ the elastic displacement field is shown to be a solution of the linear elastic wave equation outside a propagating traction free crack. The theoretical treatment shows that the numerical simulation for the $\epsilon>0$ case will converge to the $\epsilon=0$ evolution (assuming perfectly accurate computation).

Simulations show that cracks follow the location of maximum energy dissipation inside the intact material. The cracks appear as traction free internal boundaries  and both simulation and theory show that they form the wake behind a moving strain concentration consistent with Linear Elastic Fracture Mechanics. 


\section*{Acknowledgments}
This material is based upon work supported by the U. S. Army Research Laboratory and the U. S. Army Research Office under Contract/Grant Number W911NF-19-1-0245. 
Portions of this research were conducted with high performance computing resources provided by Louisiana State University (http://www.hpc.lsu.edu). 

\begin{figure}[ht!]
    \centering
    \includegraphics[width=0.49 \linewidth]{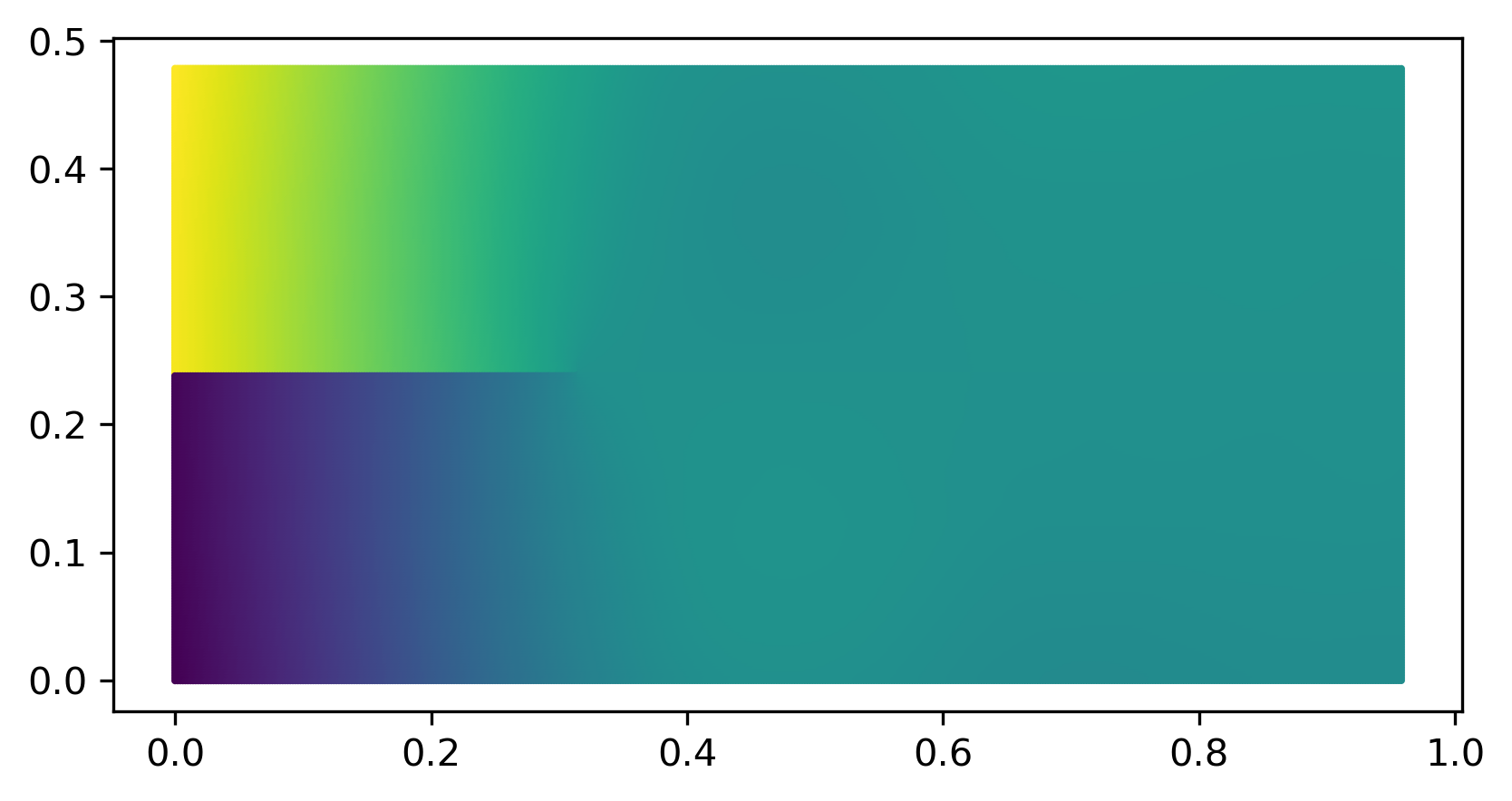}
    \includegraphics[width=0.49 \linewidth]{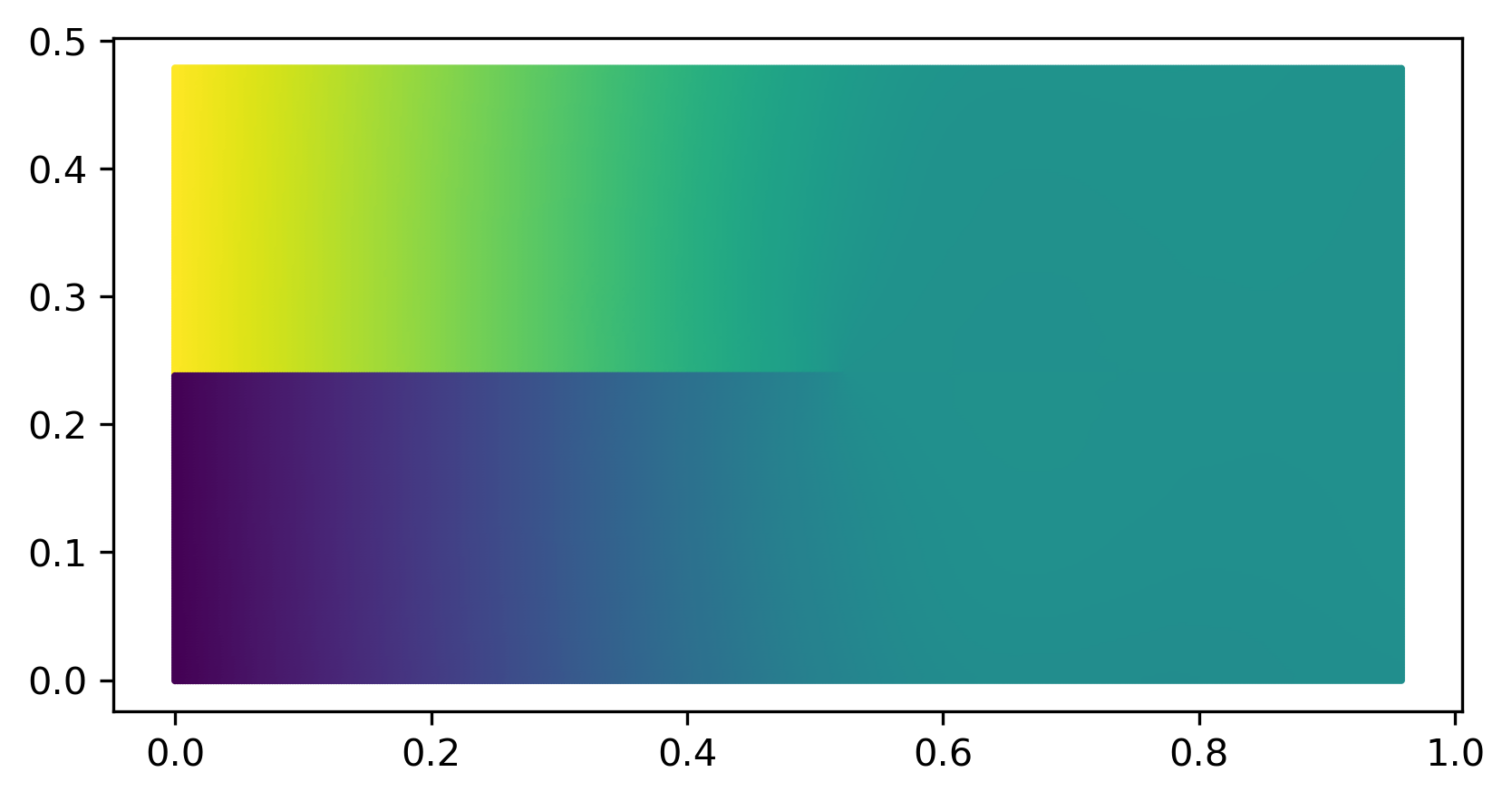}
    
    \includegraphics[width=0.49 \linewidth]{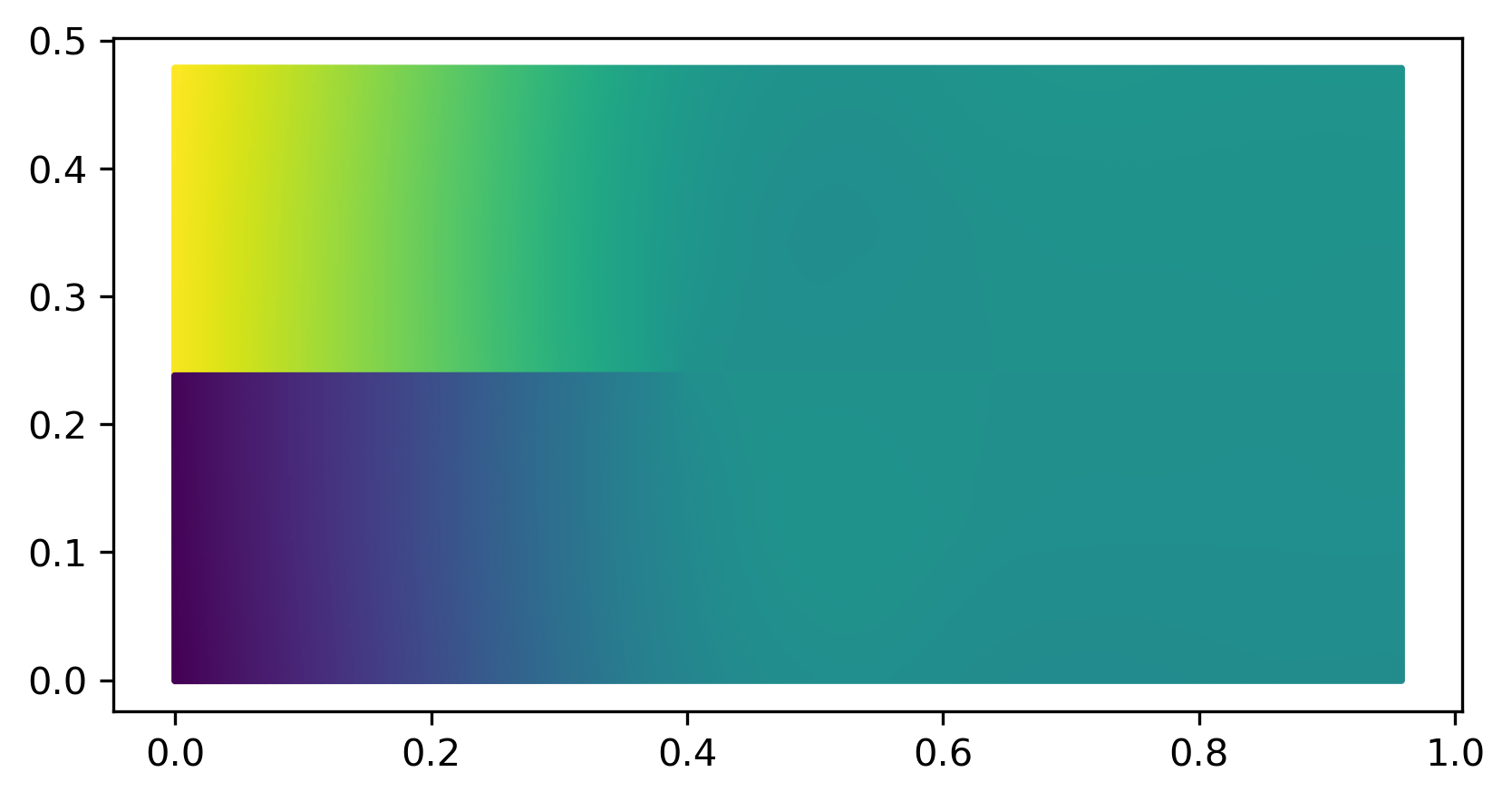}
    \includegraphics[width=0.49 \linewidth]{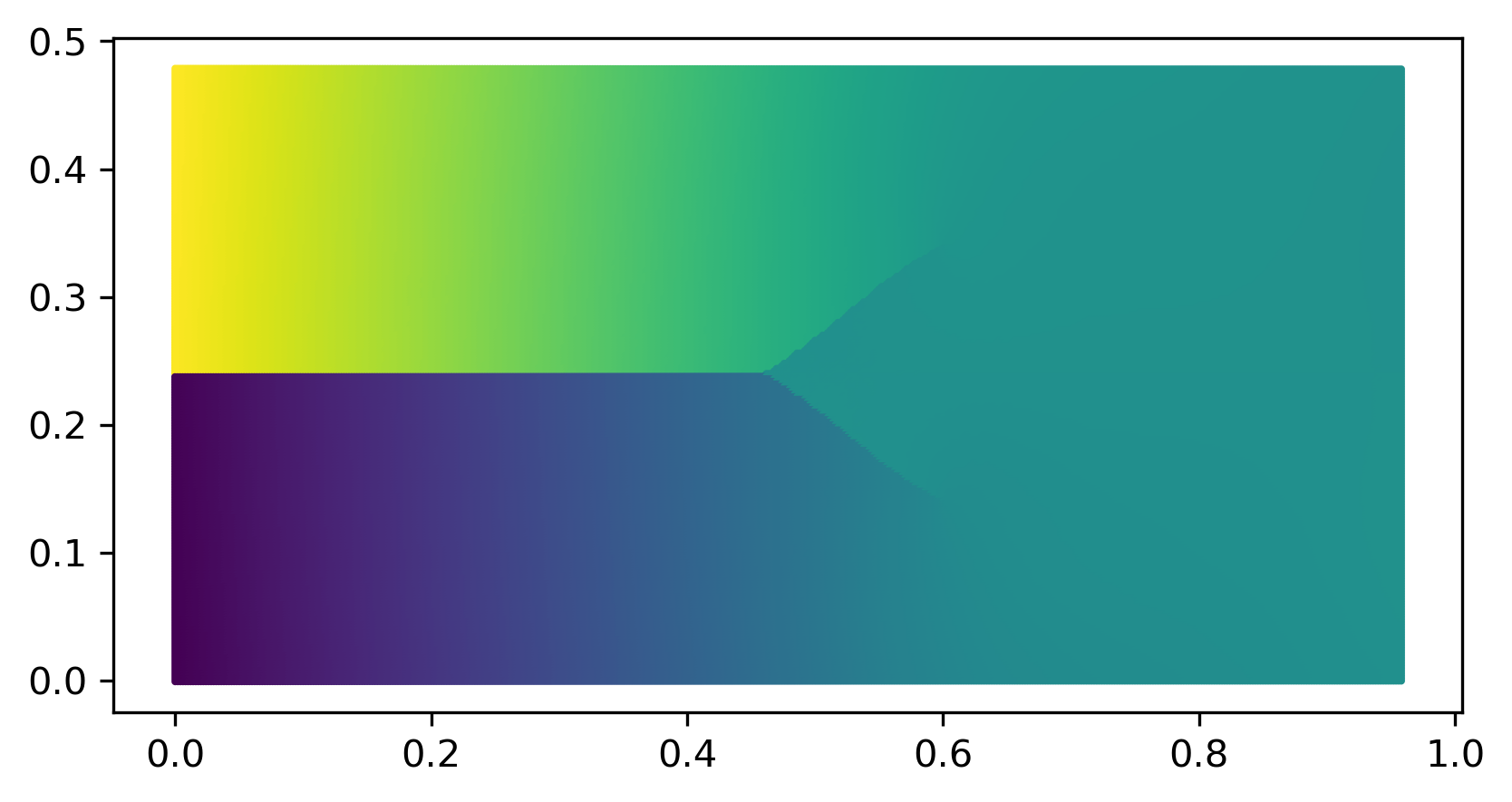}
    \caption{
    Displacement $u_y$ is shown for two cracks at t = 450, 700 $\mu$s. 
    The jump set in the displacement field reveals the evolving crack.}
    \label{fig:uy}
\end{figure}

\begin{figure}[ht!]
    \centering
    \includegraphics[width=0.49 \linewidth]{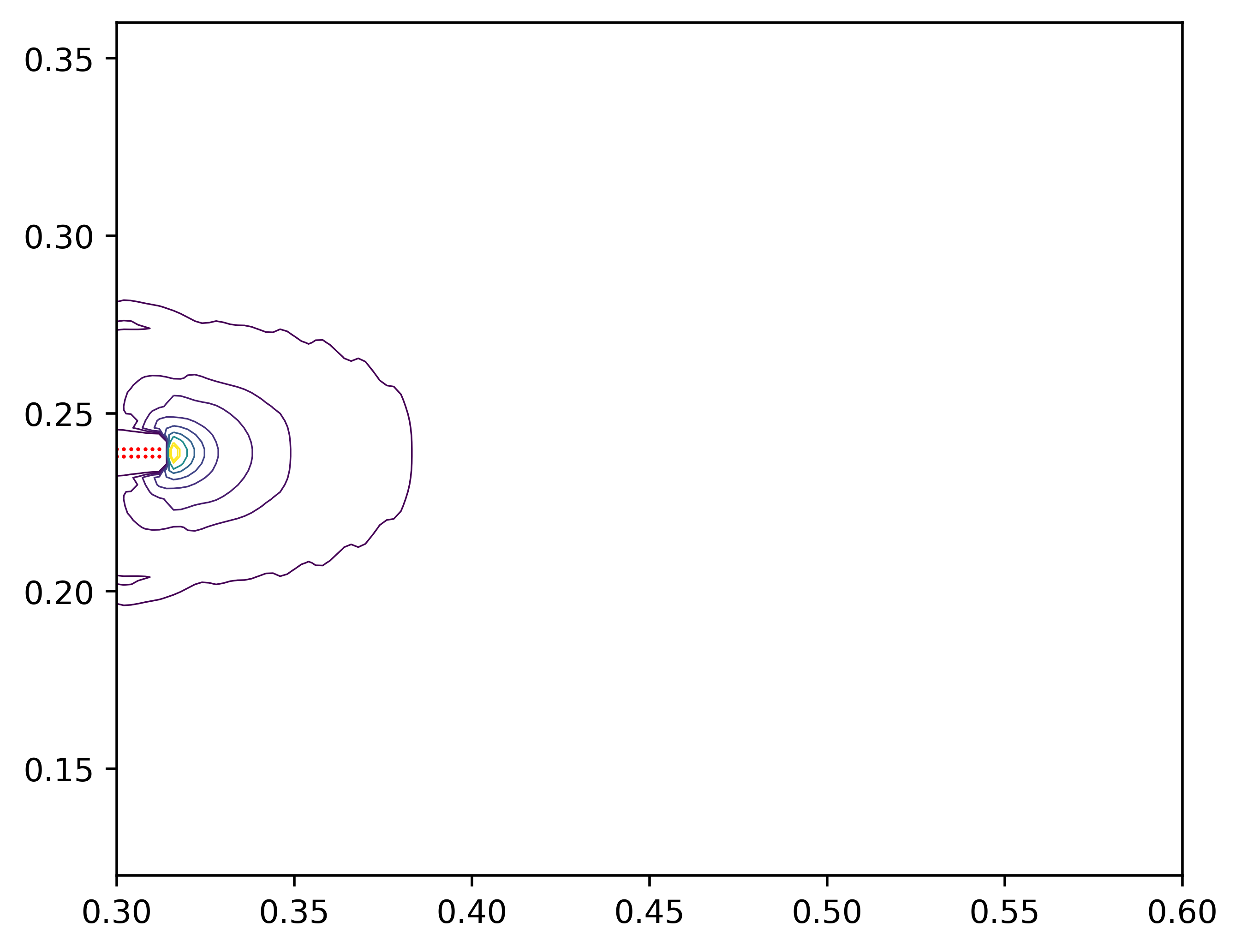}
    \includegraphics[width=0.49 \linewidth]{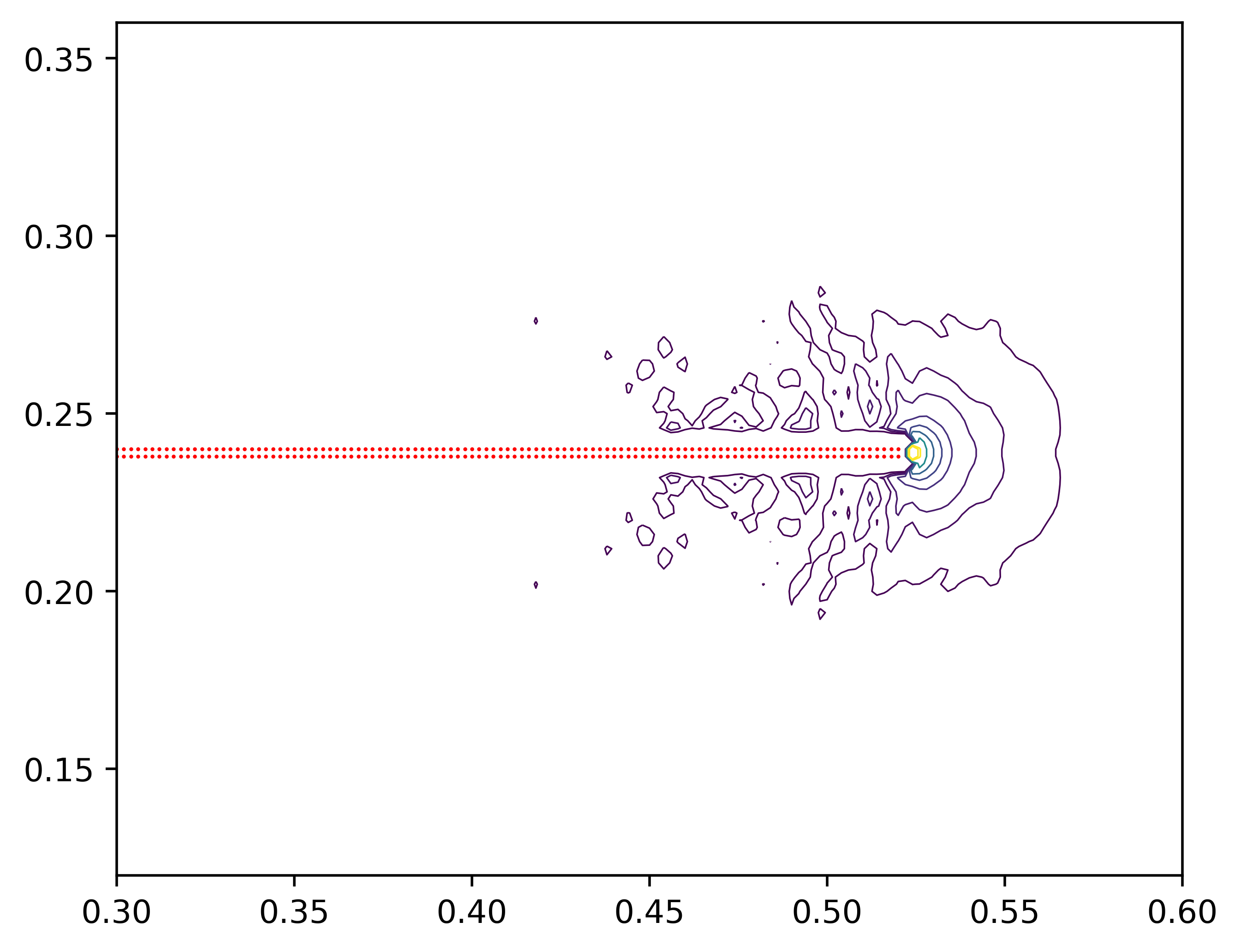}
    
    \includegraphics[width=0.49 \linewidth]{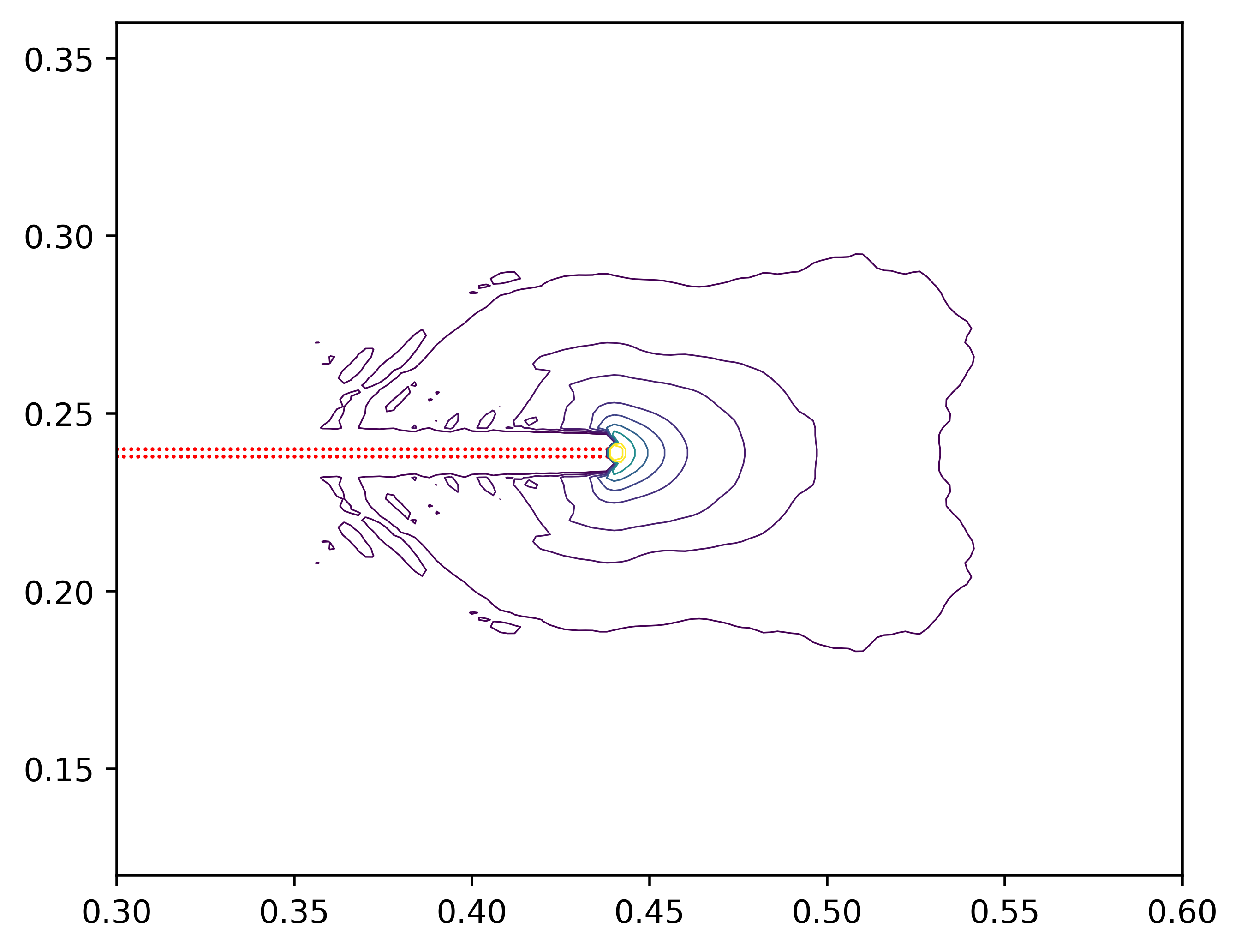}
    \includegraphics[width=0.49 \linewidth]{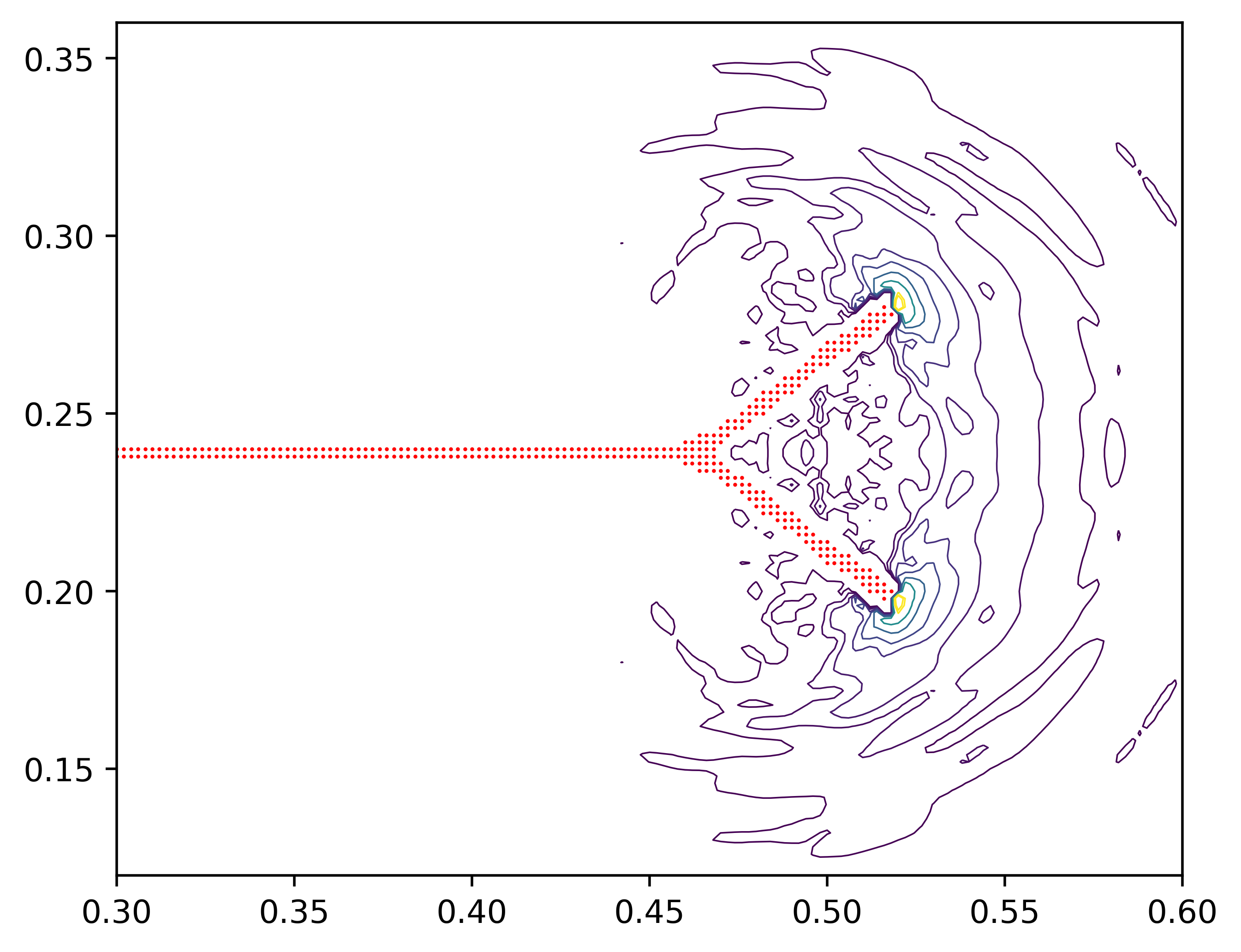}
    
    \caption{The contour lines of strain energy density $W^\epsilon(\xx,\uu)$ in front of the crack tip are shown. For the straight crack, snapshots at time t = 450, 700 $\mu$s are shown. For the bifurcating crack, we show the snapshots at time $t = 500, 600 \mu$s. The red color represents the crack path. Level curves of strain energy density over the entire domain show maximum strain energy is achieved directly in front of the crack tip.}
    \label{fig:contour-w}
\end{figure}

\begin{figure}[ht!]
    \centering
    \includegraphics[width=0.49 \linewidth]{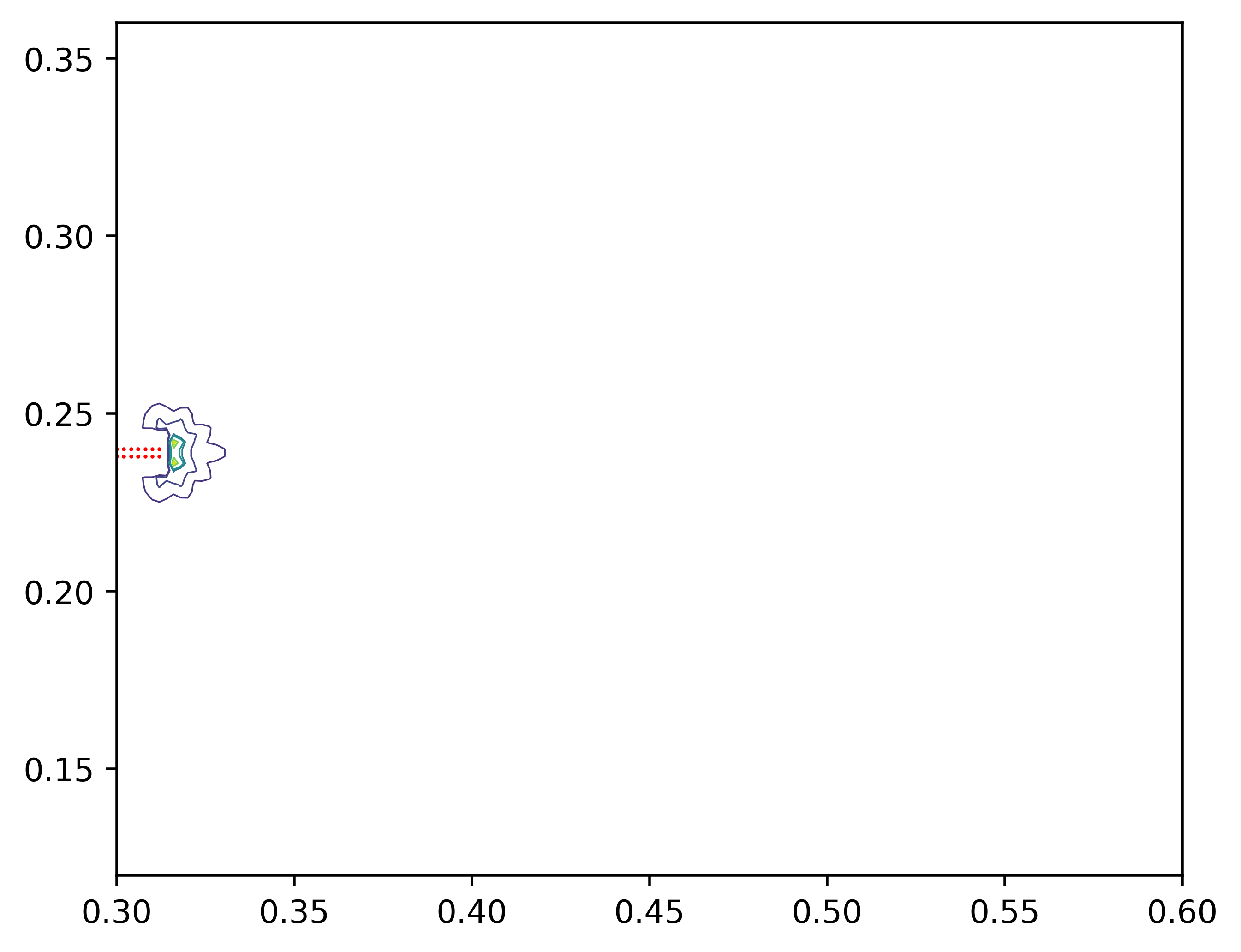}
    \includegraphics[width=0.49 \linewidth]{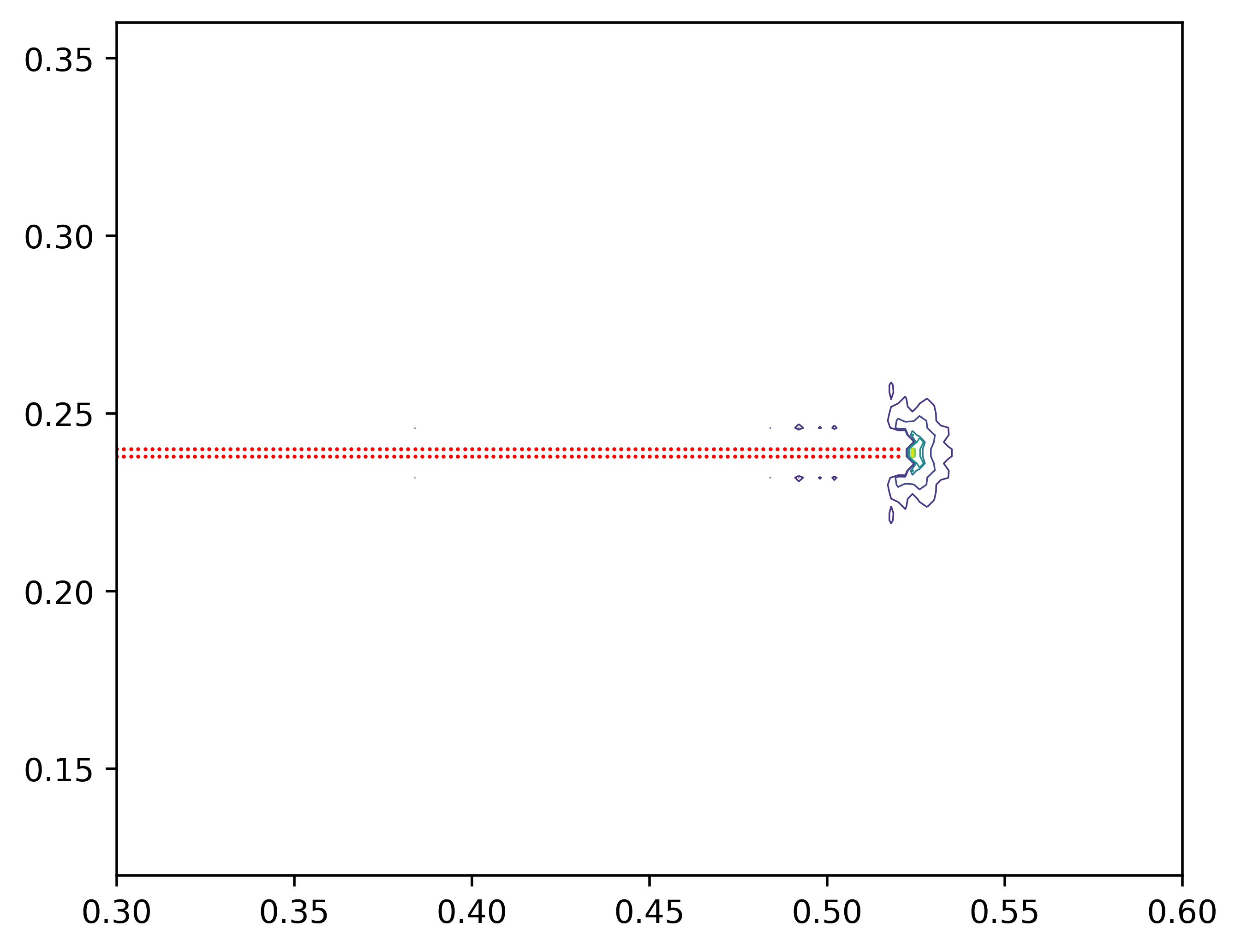}
    
    \includegraphics[width=0.49 \linewidth]{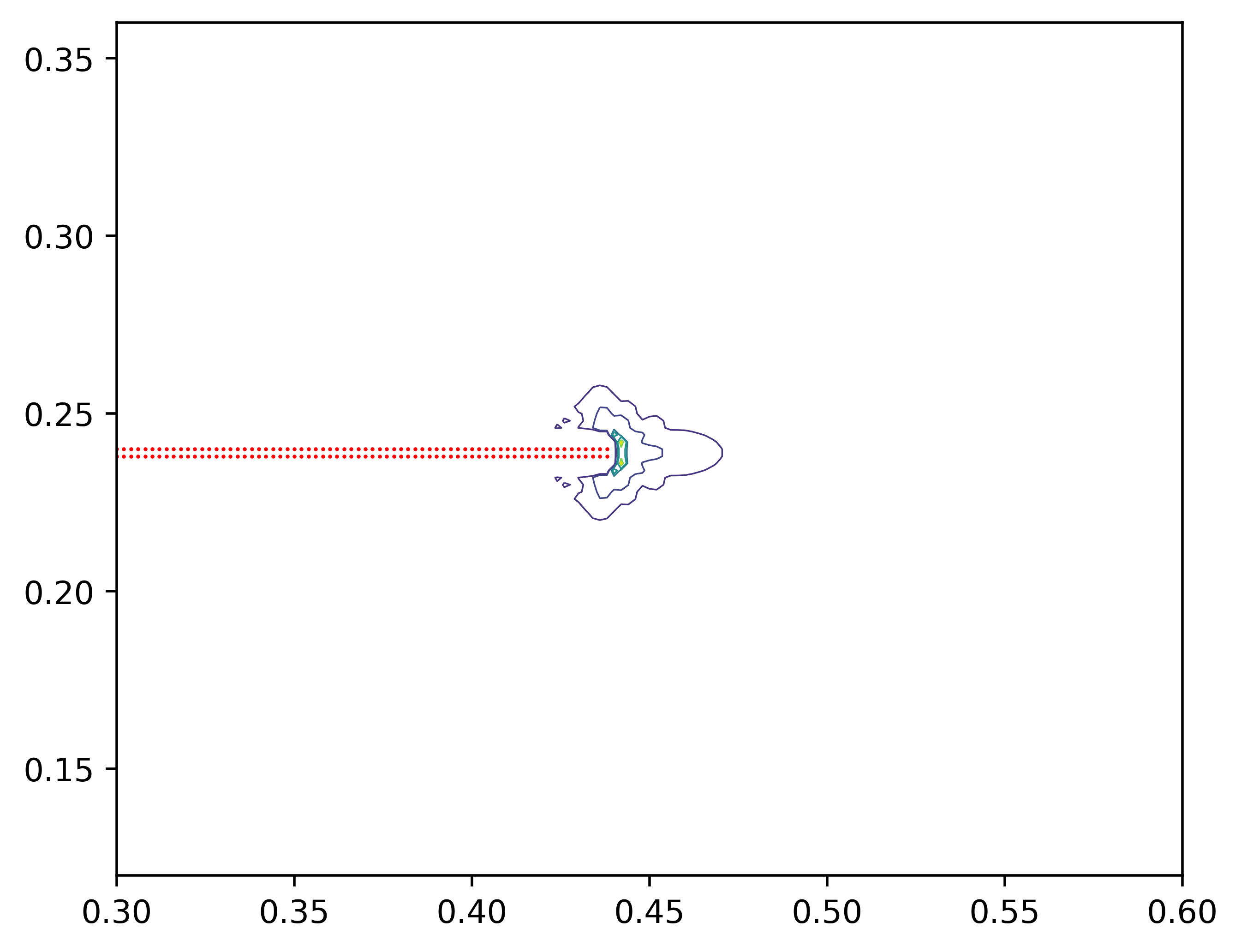}
    \includegraphics[width=0.49 \linewidth]{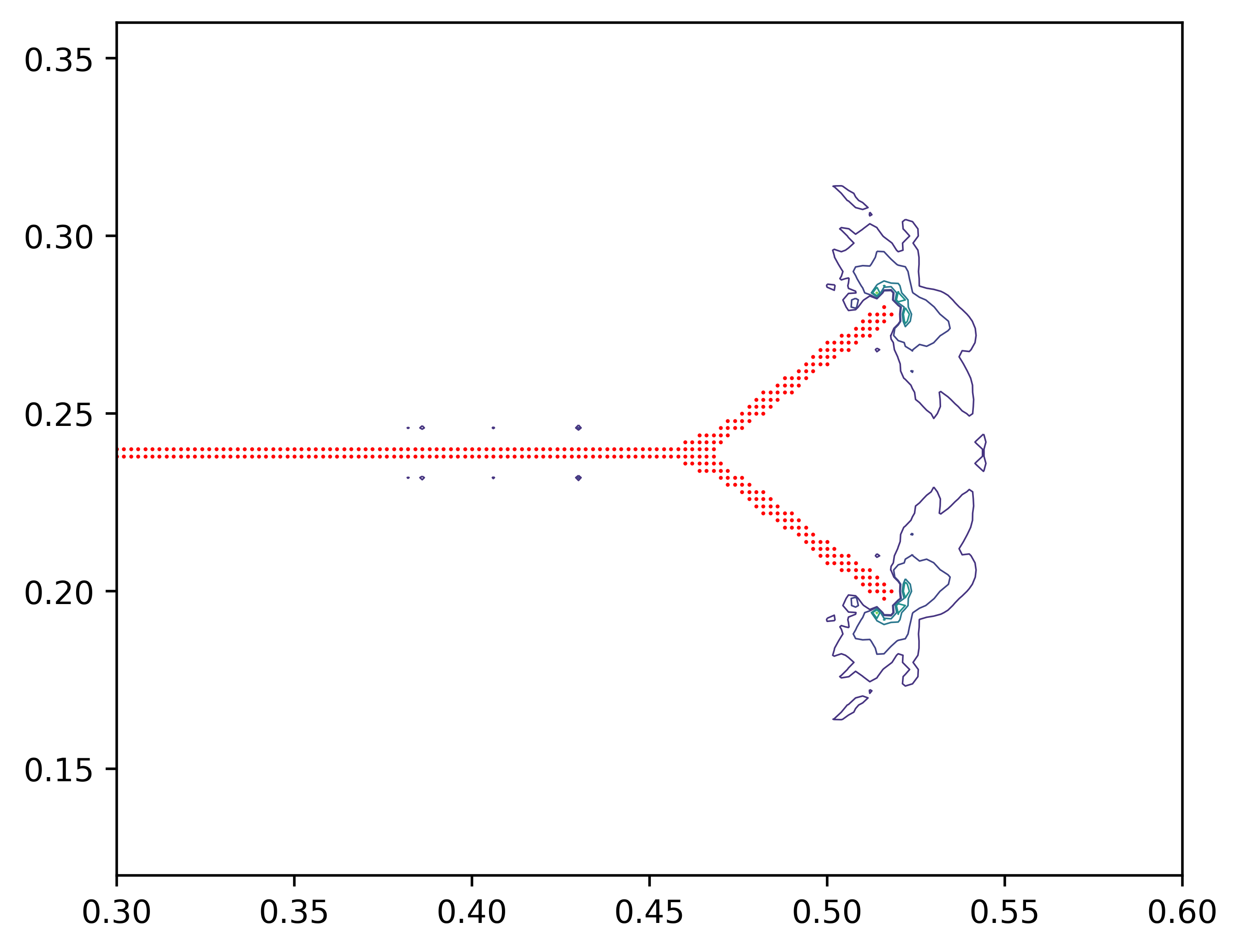}
    
    \caption{The contour lines of strain concentration $Z^\epsilon(\xx,\uu)$ in the intact material are shown, i.e. for points with $0 \le Z \le 1$. For the straight crack, snapshots at time t = 450, 700 $\mu$s are shown. For the bifurcating crack, we show the snapshots at time $t = 500, 600 \mu$s. The red color represents the crack path.}
    \label{fig:contour-z}
\end{figure}


%
%

\bibliographystyle{abbrv}
\bibliography{references}   

\end{document}